\documentclass[11pt]{article}

\usepackage{amsmath, amssymb, amsthm, mathrsfs, mathtools}
\usepackage[ruled,vlined]{algorithm2e}
\usepackage[margin=1in]{geometry}
\usepackage{float}
\usepackage[round]{natbib}
\usepackage{graphicx}
\usepackage[title]{appendix}
\usepackage{tikz-cd}
\usepackage{xcolor}
\usepackage{xurl}
\usepackage{hyperref}

\newtheorem{theorem}{Theorem}
\newtheorem{lemma}{Lemma}
\newtheorem{proposition}{Proposition}
\newtheorem{corollary}{Corollary}
\newtheorem{assumption}{Assumption}

\makeatletter
\newcommand\blfootnotenomarker[1]{
  \begingroup
  \renewcommand\thefootnote{}
  \renewcommand\@makefnmark{}
  \NoHyper\footnotetext{#1}\endNoHyper
  \endgroup
}
\makeatother

\title{The Newton--Muon Optimizer}
\author{Zhehang Du \and Weijie Su}
\date{University of Pennsylvania\\[1.5em]\today}

\newcommand{\R}{\mathbb{R}}
\newcommand{\vecop}{\mathrm{vec}}

\begin{document}
\maketitle

\begin{abstract}

The Muon optimizer has received considerable attention for its strong performance in training large language models, yet the design principle behind its matrix-gradient orthogonalization remains largely elusive. In this paper, we introduce a surrogate model that not only sheds new light on the design of Muon, but more importantly leads to a new optimizer. In the same spirit as the derivation of Newton's method, the surrogate approximates the loss as a quadratic function of the perturbation to a weight matrix $W$ using only three matrices: the gradient $G$, an output-space curvature matrix $H$, and the data matrix $Z$ that stacks the layer inputs. By minimizing this surrogate in one step and adopting a certain isotropic assumption on the weights, we obtain the closed-form update rule (up to momentum and weight decay)
\[
W \leftarrow W - \eta \cdot \mathrm{msgn}(G(ZZ^\top)^{-1}),
\]
where $\eta$ is the learning rate and $\mathrm{msgn}(X)=UV^\top$ if $X=USV^\top$ is a compact singular value decomposition. This new optimization method, which we refer to as Newton--Muon, shows that standard Muon can be interpreted as an implicit Newton-type method that neglects the right preconditioning induced by the input second moment. Empirically, on a reproduction of the earliest publicly released Modded-NanoGPT speedrun configuration using Muon for GPT-2 pretraining, Newton--Muon reaches the target validation loss in 6\% fewer iteration steps and reduces wall-clock training time by about 4\%.

\end{abstract}

\blfootnotenomarker{Emails: \texttt{\{duz,suw\}@wharton.upenn.edu}.}
\blfootnotenomarker{Code is available at \url{https://github.com/zhehangdu/Newton-Muon}.}

\section{Introduction}

Since early 2025, there has been a surge of interest in matrix-structured optimization methods for training deep neural networks and large language models (LLMs). A prominent optimizer at the center of this flurry of research activity is Muon \citep{jordan2024muon}, closely related to spectral gradient descent \citep{carlson2015preconditioned}. Specifically, letting \(f({W})\) be the loss function and \({W}\) be a layer weight matrix, we consider the optimization problem
\[
\min_{W \in \mathbb{R}^{m \times n}} f(W).
\]
Let \({G} \coloneqq \nabla_{{W}} f({W})\in\mathbb{R}^{m\times n}\) be the gradient matrix or, in practice, an approximation computed from a mini-batch. Writing the matrix sign $\mathrm{msgn}({G}) \coloneqq {U}{V}^\top$ if ${G}={U}{S}{V}^\top$ is a compact singular value decomposition (SVD), Muon updates the weight matrix as\footnote{In practice, Muon approximates \(\mathrm{msgn}({G})\) using Newton--Schulz iterations and momentum is applied to the gradient.}
\[
W \leftarrow W - \eta \cdot \mathrm{msgn}({G})
\]
for some learning rate $\eta$. Compared to AdamW~\citep{kingma2014adam,loshchilov2017decoupled}, Muon has been reported to offer a faster convergence rate and lower wall-clock time to reach the same level of loss across a broad range of model sizes~\citep{liu2025muon,essentialai2025practical,wen2025fantastic}. Furthermore, it has been used to train state-of-the-art open-source models~\citep{kimiteam2026kimik25visualagentic,zeng2026glm}, and many extensions of Muon have been proposed~\citep{li2025normuon,pethick2025training,ahn2025dion2,he2025low,xu2026fismo,qi2026delving,gu2026mano}.
 
While matrix-based optimizers have demonstrated highly effective empirical performance, the theoretical mechanisms underlying Muon remain largely mysterious. For instance, it is natural to ask why preserving the matrix structure of the gradient is beneficial and why simply discarding its singular values is empirically effective. This is, however, not entirely surprising, as the nonconvex nature of deep learning optimization makes it notoriously difficult to analyze. Given that a rigorous theoretical understanding remains out of reach, a practical approach is to establish an intuitive yet principled framework for designing deep learning optimizers~\citep{bernstein2024old,pethick2025training,lau2025polargrad,gong2026aro}. Among these, \citet{su2025isotropic} introduced the isotropic curvature model as a surrogate for approximating the loss function by assuming isotropic curvature for preconditioning and isotropic input activations. A one-step descent analysis applied to this model suggests that the optimal update direction naturally preserves the matrix subspace of the gradient, thereby partially justifying gradient orthogonalization. 
 
However, the unspecified curvature function in the isotropic curvature model limits its utility for deriving practical optimization methods for LLM training. In this paper, we address this challenge by introducing a more tractable surrogate model for understanding these optimizers and, more importantly, for proposing a new method. Taking $Q\in\R^{m\times n}$ to be a potential update direction and letting $H$ denote a curvature matrix,\footnote{Here, the curvature matrix $H$ is not the full parameter-space Hessian with respect to the vectorized weights. In particular, the curvature matrix $H$ has size $m \times m$.} this paper introduces a surrogate model of the form
\begin{equation}\label{eq:triple_intro}
f(W - Q) - f(W) \approx -\mathrm{tr}({Q}{G}^\top)
  + \frac{1}{2N} \mathrm{tr} \Big({H} {Q} ({Z}{Z}^\top) {Q}^\top\Big),
\end{equation}
where $Z = [\boldsymbol{z}_1, \ldots, \boldsymbol{z}_N]$ denotes the collection of all activation inputs from the $N$ training data points to the layer $W$. Because it involves three components---namely $G$, $H$, and $Z$---this model is referred to as the \textit{triplet quadratic surrogate model}. By design, the gradient matrix $G$ and the update direction $Q$ naturally retain their matrix forms. The linear term $-\mathrm{tr}({Q}{G}^\top)$ captures the first-order approximation, as in the isotropic curvature model, while the quadratic term $\frac{1}{2N} \mathrm{tr} \Big({H} {Q} ({Z}{Z}^\top) {Q}^\top\Big)$ is considerably simpler than the curvature function in the isotropic curvature model. Notably, this quadratic term is the matrix-form expression induced by the Kronecker-factored curvature approximation in K-FAC~\citep{martens2015optimizing}.
 
This triplet quadratic surrogate for approximating $f(W - Q) - f(W)$ is essentially of the form ``$-\textit{linear gradient term} + \textit{quadratic curvature term}$'', which closely resembles the derivation of Newton's method. As such, this surrogate model serves as a means of deriving a Newton-type method in a manner that fully leverages the matrix structure of the gradient via a one-step descent analysis; that is, by minimizing the right-hand side of \eqref{eq:triple_intro} over $Q$. In fact, although second-order methods such as Newton's method are rarely used in training large-scale neural networks, recent work demonstrates a $3$--$5\times$ iteration speedup when using the Gauss--Newton method compared to common optimizers in deep learning \citep{abreu2025potential}. Thus, there is strong motivation to develop an implicit Newton-type method for LLM training, provided that the per-iteration computational cost remains roughly comparable to that of AdamW or Muon.
 
Despite the ease of minimizing the triplet quadratic surrogate model, a key difficulty arises because the optimal solution involves the unknown curvature matrix $H$. To circumvent this, we make an assumption on the displacement $W - W^\star$, which is the difference between the current weight matrix $W$ and an optimal weight matrix $W^\star$.\footnote{Since there are generally exponentially many optimal weight matrices, we simply consider one that is locally closest to $W$.} We assume that this displacement is isotropic, in the sense that no direction is favored over another on average. Surprisingly, under this \textit{least-informative} assumption, we obtain a closed-form expression for the Newton-type update: $Q^\star \propto \mathrm{msgn}(G(ZZ^\top)^{-1})$. Accordingly, our new method updates the weight matrix according to the rule\footnote{In practice, momentum and other common implementation tricks are also used.}
\[
W \leftarrow W - \eta \cdot \mathrm{msgn}(G(ZZ^\top)^{-1}).
\]
This method differs from standard Muon in that the gradient matrix $G$ is right-preconditioned by the inverse second moment of the input data. This can be interpreted as incorporating the geometry of the data distribution to yield an update direction distinct from that of Muon. To implement Newton--Muon efficiently in practice, we maintain a running estimate of the second moment matrix $ZZ^\top$ and recompute both this estimate and its inverse only periodically, rather than at every optimization step. Between refreshes, the cached inverse is reused. We compute the damped inverse $(ZZ^\top+\gamma I_n)^{-1}$ via a Cholesky factorization followed by triangular solves. Even for very large matrices, this computational cost is no more than about $10\times$ that of a single matrix multiplication on a modern GPU. This step is then followed by a Newton--Schulz polynomial approximation to the matrix sign, such as the Polar Express method~\citep{amsel2025polar} or Gram Newton--Schulz~\citep{GramNewtonSchulz}.
 
Because this optimization method is derived in the same spirit as Newton's method, we call it \textit{Newton--Muon}. Standard Muon can be recovered as a special case when the input data's second moment is approximately isotropic, that is, $ZZ^\top \propto I_n$. Put differently, Muon can be viewed as an implicit Newton-type method that does not account for the geometry of the input data. However, this simplification is not consistent with our observation that $ZZ^\top$ is highly anisotropic in practice.
 
The right-preconditioning by the input data distribution in the update rule endows Newton--Muon with its empirical advantage over standard Muon. Indeed, compared to our reproduction of the earliest publicly logged Modded-NanoGPT speedrun~\citep{modded_nanogpt_2024} configuration using Muon, Newton--Muon achieves a notable $6\%$ reduction in the number of iterations required to reach the target validation loss, as shown in Figure~\ref{fig:record4-baseline-plot}. In terms of wall-clock time, replacing standard Muon with Newton--Muon yields a reduction of about $4\%$.

\begin{figure}[!htb]
  \centering
  \includegraphics[width=\linewidth]{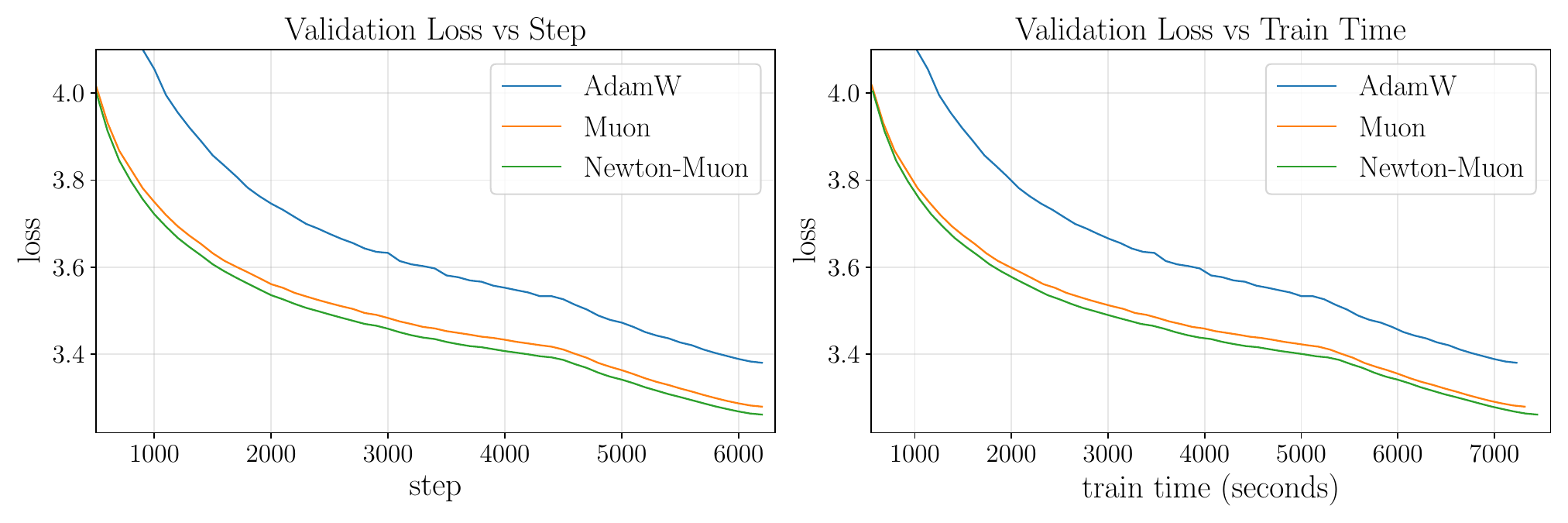}

  \includegraphics[width=\linewidth]{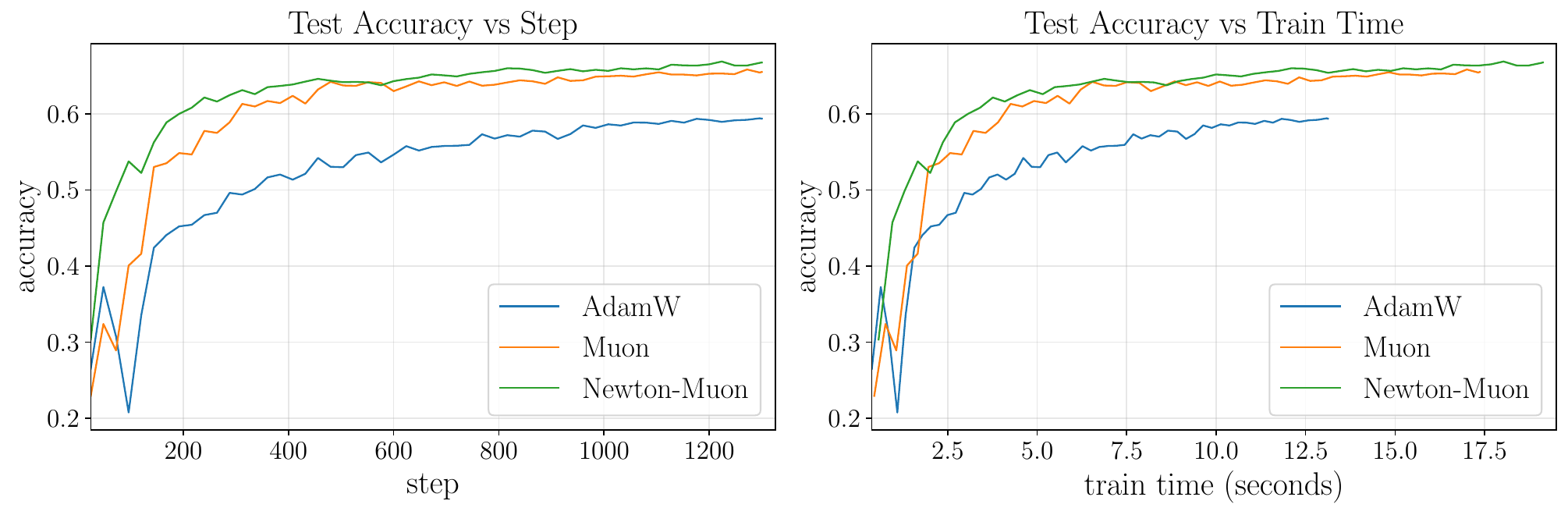}
  \caption{Top: short track Record~\#4 validation loss comparison on the Modded-NanoGPT speedrun benchmark. Record~\#4 is the earliest publicly released configuration using Muon, and our reproduction on a single H100 GPU is denoted Muon. Newton--Muon adds the activation right-preconditioner before the Newton--Schulz iterations. Newton--Muon reaches the Muon baseline final validation loss in $6\%$ fewer steps; despite a $1.8\%$ higher per-step cost from right-preconditioning, it reduces wall-clock time to that loss by about $4\%$. Bottom: CIFAR-10 experiments (Appendix~\ref{app:cifar10-details}) on a 32-layer residual MLP show that Newton--Muon outperforms both Muon and AdamW in both per-step efficiency and overall wall-clock time.}
  \label{fig:record4-baseline-plot}
\end{figure}

\subsection{Structure of the Paper}

Section~\ref{sec:exact-and-transformer} introduces the triplet quadratic surrogate model and derives the Newton--Muon update. Section~\ref{sec:single-spike-convergence} analyzes Newton--Muon on a simple quadratic case study. Section~\ref{sec:numerical-sim} develops a one-step analysis with numerical experiments under spiked activation. Section~\ref{sec:real-llm-exp} presents our experiments on LLMs. Section~\ref{sec:conclusion} discusses limitations and open directions. Appendix~\ref{app:llm-details} provides the LLM training configurations. Appendix~\ref{app:opt-compute} details efficient computation of the activation second moment inverse. Appendix~\ref{app:cifar10-details} provides the CIFAR-10 training configurations. Appendix~\ref{app:kronecker-factor} derives the Kronecker-factored curvature. Appendix~\ref{app:quadratic-score} gives quadratic score formulas under isotropic activation. Appendix~\ref{app:llm-exp-sigmaW} discusses the non-isotropic assumption.

\subsection{Related Work}
\paragraph{Matrix-based optimizers.}
Existing matrix-based optimizers can be grouped into several broad directions. One line uses matrix structure for preconditioning, curvature approximation, or adaptive updates, including spectral descent~\citep{carlson2015preconditioned}, K-FAC~\citep{martens2015optimizing,george2018fast}, Shampoo~\citep{gupta2018shampoo,morwani2024new}, and SOAP~\citep{vyas2024soap}, which performs Adafactor-style~\citep{shazeer2018adafactor} updates. Another line improves training efficiency through approximations, for example by combining preconditioning with variance reduction, as in MARS~\citep{yuan2024mars}, or by using low-rank gradient projections, as in GaLore~\citep{zhao2024galore,su2025galore}, or by accelerating matrix-function evaluation, as in PRISM~\citep{yang2026prism}. More recent work studies matrix geometry itself as the design principle, including LMO-based methods such as Scion~\citep{pethick2025training}, cheaper Muon variants such as Dion2~\citep{ahn2025dion2}, adaptive choices of learning rate such as PolarGrad~\citep{lau2025polargrad}, and rotated coordinate updates such as ARO~\citep{gong2026aro}.

\paragraph{Understanding Muon.}
A common view is that Muon implements normalized steepest descent under the spectral norm~\citep{bernstein2024old,crawshaw2025exploration,riabinin2025gluon}, which is further generalized by Lion-$\mathcal{K}$~\citep{chen2025muon}. We instead connect Muon to the classical Euclidean Newton step under a local quadratic surrogate, providing a complementary explanation for its strong performance. Related theory studies Muon's implicit bias and geometric interpretation~\citep{fan2025implicit,pethick2025training}, including simplicity bias~\citep{dragutinovic2026use}, separates curvature and gradient anisotropy~\citep{lau2025polargrad}, and analyzes convergence in stochastic nonconvex settings and structured models~\citep{li2025muon,sato2025analysis,ma2026preconditioning}, including under heavy-tailed gradient noise~\citep{yu2026sign}. Beyond asymptotic results, local one-step analysis has also been studied~\citep{su2025isotropic,davis2025spectral,gonon2026insights}.

\paragraph{Kronecker-factored curvature.}
Layerwise curvature factorizations and matrix preconditioners have been widely studied, notably in K-FAC~\citep{martens2015optimizing} and Shampoo~\citep{gupta2018shampoo}. We start from the same structural observation: for a given layer, the parameter-space curvature of a local quadratic model can be approximated by the Kronecker product $({Z}{Z}^\top/N)\otimes {H}$, where ${Z}{Z}^\top/N$ is the activation second moment and ${H}$ is an output-space curvature factor. K-FAC explicitly estimates ${H}$ as a generalized Gauss--Newton/Fisher factor from gradient second moments and then applies the two-sided preconditioner ${H}^{-1}{G}({Z}{Z}^\top)^{-1}$. In contrast, we show that Muon's matrix sign can be interpreted as an implicit left preconditioner that approximates the effect of ${H}^{-1}$.

\section{Derivation of Newton--Muon}
\label{sec:exact-and-transformer}
In this section, we first propose an amenable surrogate function for approximating the objective function we wish to minimize, followed by the derivation of the Newton--Muon method via a \textit{one-step} minimization of this surrogate. We then draw a connection between Newton's method and standard Muon to provide a principled interpretation of the latter. 

\subsection{Triplet Quadratic Surrogate}

Starting from the current iterate, we estimate the loss change induced by a candidate update direction using a local quadratic model, in the same spirit as the derivation of Newton's method. 

Assume the loss function $f:\R^{m\times n}\to\R$ is twice continuously differentiable, and let ${G} \coloneqq \nabla_{{W}} f({W})\in\R^{m\times n}$ be the gradient matrix at the current iterate ${W}\in\R^{m\times n}$. For each sample $i$, let $\boldsymbol{z}_i\in\R^n$ denote the layer input, which is the activated output from the previous layer, and write ${Z}=[\boldsymbol{z}_1,\dots,\boldsymbol{z}_N]\in\R^{n\times N}$. For a candidate update direction ${Q}\in\R^{m\times n}$, the layer output for sample $i$ changes by $-{Q}\boldsymbol{z}_i\in\R^m$. To motivate the quadratic term, we write the averaged loss as $f({W})=(1/N)\sum_{i=1}^N L_i({W}\boldsymbol{z}_i)$. Under the perturbation ${W}\mapsto {W}-{Q}$, the $i$-th output changes from ${W}\boldsymbol{z}_i$ to ${W}\boldsymbol{z}_i-{Q}\boldsymbol{z}_i$. By the integral remainder form of Taylor's theorem,
\[
L_i({W}\boldsymbol{z}_i-{Q}\boldsymbol{z}_i)
=
L_i({W}\boldsymbol{z}_i)
-
\nabla L_i({W}\boldsymbol{z}_i)^\top ({Q}\boldsymbol{z}_i)
+
h({Q}\boldsymbol{z}_i,{W}\boldsymbol{z}_i),
\]
where
\[
h({Q}\boldsymbol{z}_i,{W}\boldsymbol{z}_i)
 \coloneqq 
({Q}\boldsymbol{z}_i)^\top
\left[
\int_0^1 (1-t) \nabla^2 L_i({W}\boldsymbol{z}_i-t {Q}\boldsymbol{z}_i) \mathrm{d}t
\right]
({Q}\boldsymbol{z}_i).
\]
We then approximate this path-dependent integral \( \int_0^1 (1-t)\nabla^2 L_i({W}\boldsymbol{z}_i-t {Q}\boldsymbol{z}_i) \mathrm{d}t \approx (1/2) H \) by a fixed average curvature matrix $H$ shared across samples, so that the contribution of sample $i$ to the quadratic term is approximated by $1/(2N)({Q}\boldsymbol{z}_i)^\top {H} ({Q}\boldsymbol{z}_i)$. Summing over the $N$ samples gives the surrogate objective
\begin{equation}\label{eq:triplet-obj-derived}
  \min_{{Q}\in\R^{m\times n}} J({Q}) =  - \mathrm{tr}({Q}{G}^\top)
+ \frac1{2N}\sum_{i=1}^N ({Q}\boldsymbol{z}_i)^\top {H} ({Q}\boldsymbol{z}_i) \coloneqq - \mathrm{tr}({Q}{G}^\top) + \frac1{2N}\mathrm{tr}\Big(HQZZ^\top Q^\top\Big).
\end{equation}
This objective is closely related to the isotropic curvature model~\citep{su2025isotropic}:
\[
\min_{{Q}\in\R^{m\times n}} -\mathrm{tr}(QG^\top) + \mathbb{E}_{\zeta} h(\|Q\zeta\|),
\]
where $\zeta$ is sampled uniformly from the unit sphere and $h$ is a univariate curvature function. The linear term is the same as in \eqref{eq:triplet-obj-derived}, but the higher-order term is modeled isotropically through a radial curvature function $h$, assumed to have super-quadratic growth. In contrast, our triplet surrogate replaces the isotropic sampling of $\zeta$ with the empirical distribution from the columns of $Z$, removes the isotropy assumption on the curvature, and explicitly captures the interaction among $H$, $Z$, and $G$ in \eqref{eq:triplet-obj-derived}.

\subsection{Minimization of the Triplet Model}
\label{subsec:triplet-induced}

Minimizing \eqref{eq:triplet-obj-derived} over $Q$ is straightforward, but it becomes practically useful only if we can establish a relationship between the curvature matrix $H$ and the gradient matrix $G$. To this end, we first establish a connection between \eqref{eq:triplet-obj-derived} and the full parameter-space second-order expansion. Define the parameter-space Hessian
${\mathcal{H}}_{{W}} \coloneqq \nabla^2_{\vecop({W})} f({W}) \in \R^{(mn)\times(mn)}$,
where $\vecop({W})\in\R^{mn}$ stacks ${W}$ into a vector using a standard convention. The second-order Taylor expansion around $\vecop({W})$ gives
\begin{equation}\label{eq:param-quad}
  f({W} - {Q})
   \approx 
  f({W})
  - \mathrm{tr}(QG^\top)
  + \frac{1}{2} \vecop({Q})^\top {\mathcal{H}}_{{W}} \vecop({Q}).
\end{equation}
For the quadratic proxy of \eqref{eq:param-quad}, we employ the approximation $\vecop({G}) \approx {\mathcal{H}}_{{W}} \vecop({W}-W^\star)$, where $W^\star$ denotes a nearby local minimizer, or more generally, a nearby reference point at which $\nabla_W f(W^\star)$ is assumed to be negligibly small.

To relate $\mathcal{H}_W$ to the curvature matrix $H$ in the surrogate model \eqref{eq:triplet-obj-derived}, we use the Kronecker-factored approximation (see its application in K-FAC~\citep{martens2015optimizing} and Appendix~\ref{app:kronecker-factor} for details):
\begin{equation}\label{eq:HW-kron}
  {\mathcal{H}}_{{W}}
  \approx
  ({Z}{Z}^\top/N)\otimes {H},
\end{equation}
where $\otimes$ denotes the Kronecker product. We then obtain 
\[
  \vecop({G})
  \approx
  \big(({Z}{Z}^\top/N)\otimes {H}\big)\vecop({W}-W^\star)
  =
  \vecop\Big({H} ({W}-W^\star) ({Z}{Z}^\top/N)\Big).
\]
This immediately implies
\[
G \approx H ({W}-W^\star) {Z}{Z}^\top/N.
\]

In light of the above, we make the following assumption.
\begin{assumption}
\label{assump:exact-newton}
We assume ${G}={H} ({W}-W^\star) {Z}{Z}^\top/N, \quad {H}\succ 0,\quad {Z}{Z}^\top\succ 0$.
\end{assumption}
The last two conditions ensure that the quadratic surrogate in~\eqref{eq:triplet-obj-derived} is strictly convex in ${Q}$ and hence has a unique minimizer. If the curvature matrix $H$ is not positive definite but merely non-degenerate, we can instead solve for a stationary point of the surrogate model.

Under this assumption, the optimal solution to the surrogate objective \eqref{eq:triplet-obj-derived} takes a closed-form expression that does not explicitly involve the unknown curvature matrix $H$.
\begin{proposition}\label{prop:exact-newton-polar}
Denote by ${\Sigma}_{{W}} \coloneqq ({W}-W^\star)({W}-W^\star)^\top \in \mathbb{R}^{m\times m}$ the displacement second moment matrix and write ${\Sigma}_{{W}}^{1/2}$ for its unique positive semidefinite square root. Then under Assumption~\ref{assump:exact-newton}, the unique minimizer ${Q}^\star$ of \eqref{eq:triplet-obj-derived} takes the form
\begin{equation}\label{eq:SigmaW-polar-update3}
  {Q}^\star = 
  {\Sigma}_{{W}}^{1/2}
  \mathrm{msgn}\Big(
    {\Sigma}_{{W}}^{1/2} {G} ({Z}{Z}^\top)^{-1}
  \Big).
\end{equation}
\end{proposition}

\begin{proof}[Proof of Proposition~\ref{prop:exact-newton-polar}]
Define the compact SVD ${W}-W^\star={U}_{Q}{S}_{Q}{V}_{Q}^\top$, where ${U}_{Q}\in\mathbb{R}^{m\times r}$ and ${V}_{Q}\in\mathbb{R}^{n\times r}$ have orthonormal columns, and ${S}_{Q}\in\mathbb{R}^{r\times r}$ is a diagonal matrix with positive entries. Then ${\Sigma}_{{W}}={U}_{Q}{S}_{Q}^2{U}_{Q}^\top$ and ${\Sigma}_{{W}}^{1/2}={U}_{Q}{S}_{Q}{U}_{Q}^\top$. From Assumption~\ref{assump:exact-newton}, we have ${G}({Z}{Z}^\top)^{-1}={H}(W-W^\star)/N$, so
\[
  {\Sigma}_{{W}}^{1/2}{G}({Z}{Z}^\top)^{-1}
  =
  \frac{1}{N}{\Sigma}_{{W}}^{1/2}{H}(W-W^\star)
  =
  \frac{1}{N}{U}_{Q}{S}_{Q}\big({U}_{Q}^\top {H} {U}_{Q}\big){S}_{Q}{V}_{Q}^\top.
\]
Since ${H}\succ 0$ and ${U}_{Q}$ has full column rank, the $r \times r$ matrix ${U}_{Q}^\top {H} {U}_{Q}$ is symmetric positive definite. Hence, ${S}_{Q}\big({U}_{Q}^\top {H} {U}_{Q}\big){S}_{Q}$ is also symmetric positive definite. Therefore, 
\[
\mathrm{msgn}\Big({\Sigma}_{{W}}^{1/2}{G}({Z}{Z}^\top)^{-1}\Big)
=
{U}_{Q}{V}_{Q}^\top.
\]
Multiplying on the left by ${\Sigma}_{{W}}^{1/2}$ gives
\[
{\Sigma}_{{W}}^{1/2}
\mathrm{msgn}\Big({\Sigma}_{{W}}^{1/2}{G}({Z}{Z}^\top)^{-1}\Big)
=
{U}_{Q}{S}_{Q}{U}_{Q}^\top {U}_{Q}{V}_{Q}^\top
=
{U}_{Q}{S}_{Q}{V}_{Q}^\top
=
{W}-W^\star.
\]
Under Assumption~\ref{assump:exact-newton}, the quadratic term in~\eqref{eq:triplet-obj-derived} is equivalent to the vectorized quadratic term in~\eqref{eq:param-quad}. Thus, the solution ${Q}^\star$ to \eqref{eq:triplet-obj-derived}, that is, the Newton direction, satisfies $\vecop({Q}^\star)={\mathcal{H}}_{{W}}^{-1}\vecop({G})={\mathcal{H}}_{{W}}^{-1}{\mathcal{H}}_{{W}}\vecop({W}-W^\star)=\vecop({W}-W^\star)$. Hence ${Q}^\star={W}-W^\star$, from which \eqref{eq:SigmaW-polar-update3} follows.
\end{proof}

\subsection{Newton--Muon}

The Newton-type update \eqref{eq:SigmaW-polar-update3} can be readily used as long as the unknown displacement second moment $\Sigma_W = (W-W^\star)(W-W^\star)^\top$ can be estimated. While the proof of Proposition~\ref{prop:exact-newton-polar} shows that $Q^\star=W-W^\star$, meaning that the right-hand side of \eqref{eq:SigmaW-polar-update3} inherently depends on $Q^\star$, a closer look reveals that $\Sigma_W$ is a much coarser object to approximate than $Q^\star$ itself. Specifically, approximating $\Sigma_W$ does not require knowledge of the right singular spaces of $W-W^\star$.

Perhaps the least-informative approximation is to assume that $\Sigma_W$ is a multiple of the identity matrix. This isotropic proxy, ${\Sigma}_{{W}}\propto {I}_m$, is plausible, especially at initialization and during the early stages of training when the columns of $W - W^\star$ are approximately independent and no specific directional structure has yet emerged. Since ${\Sigma}_{{W}} = ({W}-W^\star)({W}-W^\star)^\top$ aggregates the second moments of the columns of ${W}-W^\star$, applying this isotropic proxy amounts to treating these unknown column directions as having no preferred orientation in aggregate. 

This yields a closed-form update direction depending only on \textit{observable} quantities, as shown in the following result.  

\begin{theorem}\label{thm:isotropic-sigma-newton-muon}
Under Assumption~\ref{assump:exact-newton} and the isotropic proxy ${\Sigma}_{{W}}\propto {I}_m$, we have
\[
{Q}^\star  \propto  \mathrm{msgn} \Big({G}({Z}{Z}^\top)^{-1}\Big).
\]
\end{theorem}

Although this is a corollary of Proposition~\ref{prop:exact-newton-polar}, we choose to present it as a theorem because it formally establishes the core method introduced in this paper. At iteration $t$, we update the weight matrix according to the rule:
\[
W_{t+1} = W_t - \eta_t \cdot \mathrm{msgn} \Big(G_t(Z_t Z_t^\top)^{-1}\Big),
\]
where $\eta_t$ denotes the learning rate. We call this method Newton--Muon because it implicitly acts as a Newton-type method, derived by minimizing a quadratic surrogate. Strictly speaking, when $m>n$, it mathematically cannot hold that ${\Sigma}_{{W}}\propto {I}_m$, since ${\Sigma}_{{W}}=({W}-W^\star)({W}-W^\star)^\top$ has rank at most $n$. Thus, the relation ${\Sigma}_{{W}}\propto {I}_m$ should be understood purely as an isotropic proxy for the unknown displacement second moment.

The following result shows that the Newton--Muon update is always a descent direction.

\begin{proposition}\label{prop:Newton--Muon-descent}
For the Newton--Muon direction $Q =\mathrm{msgn}\big(G(ZZ^\top)^{-1}\big)$, we have $\mathrm{tr}(G^\top Q) \ge 0$.
\end{proposition}

\begin{proof}[Proof of Proposition~\ref{prop:Newton--Muon-descent}]
Let ${G}_{\mathrm{r}} \coloneqq G(ZZ^\top)^{-1}$, and write its compact SVD as ${G}_{\mathrm{r}}={U}_{G}{S}_{G}{V}_{G}^\top$. Then $Q=\mathrm{msgn}({G}_{\mathrm{r}})={U}_{G}{V}_{G}^\top$. Since $G={G}_{\mathrm{r}}ZZ^\top$, we have
\[
\mathrm{tr}(G^\top Q)
=
\mathrm{tr}\big(( {G}_{\mathrm{r}}ZZ^\top )^\top Q\big)
=
\mathrm{tr}\big(ZZ^\top {G}_{\mathrm{r}}^\top Q\big)
=
\mathrm{tr}\big(ZZ^\top {V}_{G}{S}_{G}{V}_{G}^\top\big).
\]
Now $ZZ^\top$ and ${V}_{G}{S}_{G}{V}_{G}^\top$ are both positive semidefinite. Therefore, the trace of their product is nonnegative, so $\mathrm{tr}(G^\top Q) \ge 0$.
\end{proof}

Making a connection to the standard Muon optimizer, we establish the following result.

\begin{corollary}[Isotropic activations recover Muon]
\label{cor:isotropic-aa-muon}
Under the assumptions of Theorem~\ref{thm:isotropic-sigma-newton-muon}, if additionally
${Z}{Z}^\top \propto {I}_n$, then
\[
{Q}^\star  \propto  \mathrm{msgn}({G}).
\]
That is, the triplet quadratic surrogate model recovers the standard Muon update in this case.
\end{corollary}

However, as we will demonstrate in Section~\ref{sec:real-llm-exp}, ${Z}{Z}^\top$ is highly anisotropic in practice. Therefore, the right preconditioner $({Z}{Z}^\top)^{-1}$, which is readily available during training, can significantly alter the singular spaces of the gradient matrix $G$. Furthermore, the added computational cost of computing $(G(ZZ^\top)^{-1})$ is insignificant compared to the empirical performance gains of Newton--Muon over standard Muon.

It is also instructive to interpret Newton--Muon through the lens of orthogonal equivariance. Standard Muon is basis-free in the sense that for any orthogonal matrices ${O}_m\in\mathbb{R}^{m\times m}$ and ${O}_n\in\mathbb{R}^{n\times n}$, we have
\[
\mathrm{msgn}({O}_m {G} {O}_n)={O}_m\mathrm{msgn}({G}){O}_n.
\]
Thus, rotating the gradient matrix on the left or right correspondingly rotates the Muon update in the exact same manner. For Newton--Muon, the analogous commutative diagram holds if the right rotation of the gradient matrix $G$ is accompanied by the coordinate transformation $Z\mapsto O_n^\top Z$. Consequently,
\[
\mathrm{msgn} \left({O}_m {G} {O}_n \big((O_n^\top Z)(O_n^\top Z)^\top\big)^{-1}\right)
=
{O}_m\mathrm{msgn}\big({G}({Z}{Z}^\top)^{-1}\big){O}_n.
\]
In contrast, if one rotates ${G}$ on the right while keeping ${Z}$ fixed, right equivariance generally fails unless ${O}_n$ commutes with ${Z}{Z}^\top$, as illustrated in Figure~\ref{fig:muon-newton-muon-equivariance}.

\begin{figure}[!htb]
\centering
\[
\begin{tikzcd}[column sep=huge, row sep=large]
G \arrow[r] \arrow[d]
& \mathrm{msgn}(G) \arrow[d] \\
{O}_m G {O}_n \arrow[r]
& {O}_m\mathrm{msgn}(G){O}_n
\end{tikzcd}
\qquad
\begin{tikzcd}[column sep=huge, row sep=large]
(G,Z) \arrow[r] \arrow[d]
& \mathrm{msgn}\big(G(ZZ^\top)^{-1}\big) \arrow[d] \\
({O}_m G {O}_n,{O}_n^\top Z) \arrow[r]
& {O}_m \mathrm{msgn}\big( G (ZZ^\top)^{-1}\big){O}_n
\end{tikzcd}
\]
\caption{Left: standard Muon is orthogonally equivariant. Right: Newton--Muon takes the pair $(G,Z)$ as input, and the diagram commutes if the right rotation of $G$ is accompanied by the transformation $Z\mapsto O_n^\top Z$.}
\label{fig:muon-newton-muon-equivariance}
\end{figure}

\section{Convergence Analysis of Newton--Muon: Case Study}
\label{sec:single-spike-convergence}

This section presents a case study of Newton--Muon in a simple quadratic model under a single spike assumption on ${Z}{Z}^\top$, with one spiked eigendirection and an isotropic bulk. A fully general convergence analysis is difficult, so we study this simple model in which explicit rates can be obtained. The main finding is that Newton--Muon achieves a contraction rate independent of the spike condition number $\kappa$, whereas the rates of gradient descent and Muon scale as $1/\kappa$ as $\kappa$ increases. We consider the iterative optimization of ${W}$ under full-batch gradient descent, Muon, and Newton--Muon, using the learning rate sequence $\{\eta_t\}_{t\ge 0}$. We consider the objective
\[
  \min_W f({W}) \coloneqq \frac12 \|{W}Z-W^\star{Z}\|_F^2,
\]
where ${Z}\in\R^{n\times N}$ is fixed, $W^\star\in\R^{m\times n}$ is the ground-truth matrix, and $\|\cdot\|_F$ denotes the Frobenius norm. The exact gradient is
\begin{equation}
  \label{eq:ss-grad}
  \nabla_{{W}} f({W}) = ({W}-W^\star){Z}{Z}^\top.
\end{equation}
Denote $W_t$ as the weight matrix at iteration $t$. The resulting updates are
\begin{align}
  \label{eq:ss-gd-update}
  \text{GD:}\quad
  W_{t+1}
  &=
  W_t-\eta_t \nabla_{W} f(W_t)
  =
  W_t-\eta_t ( W_t-W^\star )ZZ^\top, \\
  \label{eq:ss-muon-update}
  \text{Muon:}\quad
  W_{t+1}
  &=
  W_t-\eta_t \mathrm{msgn}\big(\nabla_{W} f(W_t)\big)
  =
  W_t-\eta_t \mathrm{msgn}\big(( W_t-W^\star )ZZ^\top\big), \\
  \label{eq:ss-Newton--Muon-update}
  \text{Newton--Muon:}\quad
  W_{t+1}
  &=
  W_t-\eta_t \mathrm{msgn}\big(\nabla_{W} f(W_t)(ZZ^\top)^{-1}\big)
  =
  W_t-\eta_t \mathrm{msgn}(W_t-W^\star).
\end{align}
Before going into the theoretical details, we can already see the potential advantage of Newton--Muon over gradient descent and Muon. From the update equations~\eqref{eq:ss-gd-update}--\eqref{eq:ss-Newton--Muon-update}, gradient descent and Muon apply the right-multiplication factor $ZZ^\top$ to the displacement $W_t-W^\star$. When $ZZ^\top$ is highly anisotropic, as in LLM training, where the activation matrix $Z$ has bounded stable rank~\citep{davis2025spectral} and a very large condition number (Table~\ref{tab:real-K-stats}), this anisotropy can distort the effective update, over-emphasizing some directions toward $W^\star$ while suppressing others. In contrast, Newton--Muon explicitly cancels this right-side anisotropy by multiplying by $(ZZ^\top)^{-1}$.

\paragraph{Single spike model.} We now introduce the single-spike model and restrict to a low-dimensional invariant subspace spanned by a single spiked eigendirection of ${Z}{Z}^\top$ and its isotropic complement.
We initialize ${W}_0=0$ and assume that the singular values of ${Z}$ satisfy $\sigma_1^2=\kappa>1$ and $\sigma_2^2=\cdots=\sigma_n^2=1$, so that
\begin{equation}
  \label{eq:ss-spike}
  {Z}{Z}^\top = {U}_{Z}\mathrm{diag}(\kappa,1,\ldots,1){U}_{Z}^\top,
  \qquad \kappa>1,
\end{equation}
for some orthogonal matrix ${U}_{Z}$. Let $\boldsymbol{e}_1\in\R^n$ denote the first column of $U_Z$. We assume \(W^\star\) has the following structured rank-$r$ form
\begin{equation}
  \label{eq:ss-init-decomp}
  W^\star
  =
  -\boldsymbol{u}_1 \big(\alpha_{1,0} \boldsymbol{e}_1^\top + \beta_{1,0} \boldsymbol{b}_1^\top\big)
  -
  \sum_{i=2}^r \boldsymbol{u}_i \big(\beta_{i,0} \boldsymbol{b}_i^\top\big),
\end{equation}
where $\boldsymbol{u}_1,\dots,\boldsymbol{u}_r\in\R^m$ are orthonormal, and $\boldsymbol{b}_1,\dots,\boldsymbol{b}_r\in\R^n$ are orthonormal with $\boldsymbol{b}_i\perp \boldsymbol{e}_1$ for all $i$. In particular, this requires $r\le m$ and $r\le n-1$. We assume $(\alpha_{1,0},\beta_{1,0})\neq(0,0)$ and $\beta_{i,0}\neq 0$ for $i\ge 2$.
This assumption ensures that (i) $W^\star$ is rank-$r$, and (ii) only the first mode mixes the spiked direction $\boldsymbol{e}_1$ with a non-spike direction; all other modes lie entirely in the isotropic subspace $\boldsymbol{e}_1^\perp$.
Since ${W}_0=0$, the initial residual ${W}_0-W^\star=-W^\star$ lies in the same family. We will show that, for all $t$,
\begin{equation}
  \label{eq:ss-residual-decomp}
  W_t-W^\star
  =
  \boldsymbol{u}_1 \big(\alpha_{1,t} \boldsymbol{e}_1^\top + \beta_{1,t} \boldsymbol{b}_1^\top\big)
  +
  \sum_{i=2}^r \boldsymbol{u}_i \big(\beta_{i,t} \boldsymbol{b}_i^\top\big).
\end{equation}

The main consequence of \eqref{eq:ss-spike} and \eqref{eq:ss-init-decomp} is that the dynamics of gradient descent, Muon, and Newton--Muon decouple across the modes $i=1,\dots,r$, yielding explicit one-dimensional recursions for the coefficients. We adopt the convention that $\mathrm{msgn}(0)=0$, $\beta/|\beta|=0$ when $\beta=0$, and $(\alpha,\beta)/\sqrt{\alpha^2+\beta^2}=(0,0)$ when $(\alpha,\beta)=(0,0)$.

\begin{lemma}[Dynamics under the single spike model]
\label{lem:ss-mode-wise}
Under \eqref{eq:ss-spike} and \eqref{eq:ss-init-decomp}, gradient descent \eqref{eq:ss-gd-update}, Muon \eqref{eq:ss-muon-update}, and Newton--Muon \eqref{eq:ss-Newton--Muon-update} preserve the decomposition \eqref{eq:ss-residual-decomp}. In particular, the dynamics decouple across modes $i=1,\dots,r$, and the coefficients satisfy the following recursions.
\begin{itemize}
\item For gradient descent,
\begin{equation}
\label{eq:ss-rec-gd}
\alpha_{1,t+1}=(1-\eta_t\kappa)\alpha_{1,t},
\qquad
\beta_{i,t+1}=(1-\eta_t)\beta_{i,t},\quad i=1,\dots,r.
\end{equation}

\item For Muon,
\begin{equation}
\label{eq:ss-rec-muon}
\alpha_{1,t+1}
=
\alpha_{1,t}
-
\eta_t \frac{\kappa\alpha_{1,t}}{\sqrt{\kappa^2\alpha_{1,t}^2+\beta_{1,t}^2}},
\qquad
\beta_{i,t+1}
=
\begin{cases}
\beta_{1,t}
-\eta_t \bigl(\beta_{1,t}/\sqrt{\kappa^2\alpha_{1,t}^2+\beta_{1,t}^2}\bigr),
& i=1, \\
\beta_{i,t}
-\eta_t \bigl(\beta_{i,t}/|\beta_{i,t}|\bigr),
& i=2,\dots,r.
\end{cases}
\end{equation}

\item For Newton--Muon,
\begin{equation}
\label{eq:ss-rec-Newton--Muon}
\alpha_{1,t+1}
=
\alpha_{1,t}
-
\eta_t \frac{\alpha_{1,t}}{\sqrt{\alpha_{1,t}^2+\beta_{1,t}^2}},
\qquad
\beta_{i,t+1}
=
\begin{cases}
\beta_{1,t}
-\eta_t \bigl(\beta_{1,t}/\sqrt{\alpha_{1,t}^2+\beta_{1,t}^2}\bigr),
& i=1, \\
\beta_{i,t}
-\eta_t \bigl(\beta_{i,t}/|\beta_{i,t}|\bigr),
& i=2,\dots,r.
\end{cases}
\end{equation}
\end{itemize}
\end{lemma}

\begin{proof}[Proof of Lemma~\ref{lem:ss-mode-wise}]
Define $\boldsymbol{v}_{1,t}\coloneqq \alpha_{1,t}\boldsymbol{e}_1+\beta_{1,t}\boldsymbol{b}_1$ and, for $i\ge 2$, $\boldsymbol{v}_{i,t}\coloneqq \beta_{i,t}\boldsymbol{b}_i$, so that
\begin{equation}
  \label{eq:ss-v-decomp}
  W_t-W^\star=\sum_{i=1}^r \boldsymbol{u}_i\boldsymbol{v}_{i,t}^\top.
\end{equation}
Since $\{\boldsymbol{u}_i\}_{i=1}^r$ are orthonormal and $\{\boldsymbol{b}_i\}_{i=1}^r$ are orthonormal with $\boldsymbol{b}_i\perp \boldsymbol{e}_1$, we have $\boldsymbol{v}_{i,t}\perp \boldsymbol{v}_{j,t}$ for $i\neq j$ for every $t$.
Under \eqref{eq:ss-spike}, $ZZ^\top \boldsymbol{e}_1=\kappa \boldsymbol{e}_1$, and the eigenspace with eigenvalue $1$ is $\boldsymbol{e}_1^\perp$, so $ZZ^\top \boldsymbol{b}_i=\boldsymbol{b}_i$ for all $i=1,\dots,r$. Hence
\begin{equation}
  \label{eq:ss-v-mapped}
  ZZ^\top \boldsymbol{v}_{1,t}=\kappa\alpha_{1,t}\boldsymbol{e}_1+\beta_{1,t}\boldsymbol{b}_1 \in \mathrm{span}\{\boldsymbol{e}_1,\boldsymbol{b}_1\},
  \qquad
  ZZ^\top \boldsymbol{v}_{i,t}=\boldsymbol{v}_{i,t}\in \mathrm{span}\{\boldsymbol{b}_i\}\ \ (i\ge 2).
\end{equation}
Moreover, the transformed vectors $\{ZZ^\top\boldsymbol{v}_{i,t}\}_{i=1}^r$ remain mutually orthogonal across $i$. Using \eqref{eq:ss-grad} and \eqref{eq:ss-v-decomp},
\begin{equation}
  \label{eq:ss-grad-decomp}
  \nabla_W f(W_t)=(W_t-W^\star)ZZ^\top
  =
  \sum_{i=1}^r \boldsymbol{u}_i\big((ZZ^\top\boldsymbol{v}_{i,t})^\top\big).
\end{equation}
In particular, by \eqref{eq:ss-v-mapped},
\[
(ZZ^\top\boldsymbol{v}_{1,t})^\top=(\kappa\alpha_{1,t})\boldsymbol{e}_1^\top+\beta_{1,t}\boldsymbol{b}_1^\top,
\qquad
(ZZ^\top\boldsymbol{v}_{i,t})^\top=\beta_{i,t}\boldsymbol{b}_i^\top \ (i\ge 2).
\]

For gradient descent, substituting \eqref{eq:ss-grad-decomp} into \eqref{eq:ss-gd-update} gives
\[
W_{t+1}-W^\star
=
(W_t-W^\star)-\eta_t(W_t-W^\star)ZZ^\top,
\]
and matching coefficients in the basis $\{\boldsymbol{u}_1\boldsymbol{e}_1^\top,\boldsymbol{u}_1\boldsymbol{b}_1^\top,\boldsymbol{u}_i\boldsymbol{b}_i^\top\}$ yields $\alpha_{1,t+1}=(1-\eta_t\kappa)\alpha_{1,t}$ and $\beta_{i,t+1}=(1-\eta_t)\beta_{i,t}$ for all $i=1,\dots,r$, i.e., \eqref{eq:ss-rec-gd}.

For Muon, since the right factors $\{ZZ^\top\boldsymbol{v}_{i,t}\}_{i=1}^r$ in \eqref{eq:ss-grad-decomp} are mutually orthogonal and the left factors $\{\boldsymbol{u}_i\}_{i=1}^r$ are orthonormal, the matrix sign of $(W_t-W^\star)ZZ^\top$ is given by
\[
\mathrm{msgn}\big((W_t-W^\star)ZZ^\top\big)
=
\sum_{i:\|ZZ^\top\boldsymbol{v}_{i,t}\|_2>0}
\boldsymbol{u}_i
\left(\frac{(ZZ^\top\boldsymbol{v}_{i,t})^\top}{\|ZZ^\top\boldsymbol{v}_{i,t}\|_2}\right).
\]
For $i\ge 2$, $ZZ^\top\boldsymbol{v}_{i,t}=\boldsymbol{v}_{i,t}$ and $\|ZZ^\top\boldsymbol{v}_{i,t}\|_2=|\beta_{i,t}|$, yielding the $i\ge 2$ case in \eqref{eq:ss-rec-muon}. For $i=1$, we have
$ZZ^\top\boldsymbol{v}_{1,t}=\kappa\alpha_{1,t}\boldsymbol{e}_1+\beta_{1,t}\boldsymbol{b}_1$ and
$\|ZZ^\top\boldsymbol{v}_{1,t}\|_2=\sqrt{\kappa^2\alpha_{1,t}^2+\beta_{1,t}^2}$, so substitution into \eqref{eq:ss-muon-update} yields the $i=1$ case in \eqref{eq:ss-rec-muon}. This also preserves \eqref{eq:ss-residual-decomp}.

For Newton--Muon, since the right factors $\{\boldsymbol{v}_{i,t}\}_{i=1}^r$ are mutually orthogonal and the left factors $\{\boldsymbol{u}_i\}_{i=1}^r$ are orthonormal, the matrix sign of $W_t-W^\star$ is obtained mode-by-mode:
\[
\mathrm{msgn}(W_t-W^\star)
=
\sum_{i:\|\boldsymbol{v}_{i,t}\|_2>0}
\boldsymbol{u}_i
\left(\frac{\boldsymbol{v}_{i,t}}{\|\boldsymbol{v}_{i,t}\|_2}\right)^\top.
\]
Substituting into \eqref{eq:ss-Newton--Muon-update} gives
\[
W_{t+1}-W^\star
=
(W_t-W^\star)-\eta_t\mathrm{msgn}(W_t-W^\star),
\]
and matching coefficients yields the $i=1$ and $i\ge 2$ cases in \eqref{eq:ss-rec-Newton--Muon}. This preserves \eqref{eq:ss-residual-decomp} and completes the proof.
\end{proof}

Lemma~\ref{lem:ss-mode-wise} reduces the matrix iterates to scalar recursions for the coefficients $\{\alpha_{1,t},\beta_{i,t}\}$. In particular, convergence of $W_t\to W^\star$ is equivalent to $\alpha_{1,t}\to 0$ and $\beta_{i,t}\to 0$ for all $i=1,\dots,r$, and the effect of the spike $\kappa$ is isolated to the mixed $(\alpha_{1,t},\beta_{1,t})$ mode.

\begin{corollary}[Convergence rates]
\label{cor:ss-rates}
Assume \eqref{eq:ss-spike} and \eqref{eq:ss-init-decomp}, and let $r_0>0$ satisfy $|\alpha_{1,0}|\le r_0$ and $|\beta_{i,0}|\le r_0$ for all $i=1,\dots,r$. Define $r_t$ recursively so that $|\alpha_{1,t}|\le r_t$ and $|\beta_{i,t}|\le r_t$ for all $i=1,\dots,r$, and at each step choose $\eta_t$ greedily to minimize the worst-case next bound $r_{t+1}$. Then, for any target $0<\varepsilon<r_0$, gradient descent, Newton--Muon, and Muon reach $|\alpha_{1,t}|\le \varepsilon$ and $|\beta_{i,t}|\le \varepsilon$ for all $i=1,\dots,r$ after at most $T_{\mathrm{GD}}(\varepsilon)=O(\kappa\log(r_0/\varepsilon))$, $T_{\mathrm{NM}}(\varepsilon)=O(\log(r_0/\varepsilon))$, and $T_{\mathrm{M}}(\varepsilon)=O(\kappa\log(r_0/\varepsilon))$ iterations, respectively. In particular, Newton--Muon converges faster than both gradient descent and Muon by a factor of order $\kappa$.
\end{corollary}

\begin{proof}[Proof of Corollary~\ref{cor:ss-rates}]
For gradient descent, \eqref{eq:ss-rec-gd} gives $|\alpha_{1,t+1}| = |1-\eta_t\kappa| |\alpha_{1,t}|$ and $|\beta_{i,t+1}| = |1-\eta_t| |\beta_{i,t}|$ for $i=1,\dots,r$. Hence, if $|\alpha_{1,t}|, |\beta_{i,t}| \le r_t$, then $r_{t+1} \le \max\{|1-\eta_t\kappa|,\ |1-\eta_t|\} r_t$. The greedy choice minimizes the right-hand side by solving $|1-\eta_t\kappa| = |1-\eta_t|$, which gives $\eta_t = 2/(\kappa+1)$ and yields $r_{t+1} = (\kappa-1)/(\kappa+1) r_t$.

For Newton--Muon, let $s_t \coloneqq \sqrt{\alpha_{1,t}^2+\beta_{1,t}^2}$. By \eqref{eq:ss-rec-Newton--Muon}, $\alpha_{1,t+1} = \bigl(1-\eta_t/s_t\bigr)\alpha_{1,t}$ and $\beta_{1,t+1} = \bigl(1-\eta_t/s_t\bigr)\beta_{1,t}$. Since $|\alpha_{1,t}|, |\beta_{1,t}| \le r_t$, we have $s_t \le \sqrt{2}\,r_t$. We first bound $\alpha_{1,t+1}$ and $\beta_{1,t+1}$. If $\eta_t \le s_t$, we obtain $1-\eta_t/s_t \le 1-\eta_t/(\sqrt{2}\,r_t)$, so $|\alpha_{1,t+1}| \le \bigl(1-\eta_t/(\sqrt{2}\,r_t)\bigr)|\alpha_{1,t}| \le r_t-\eta_t/\sqrt{2}$, and similarly $|\beta_{1,t+1}| \le r_t-\eta_t/\sqrt{2}$. If instead $\eta_t \ge s_t$, then $|1-\eta_t/s_t|=\eta_t/s_t-1$, and therefore $|\alpha_{1,t+1}| = (\eta_t/s_t-1)|\alpha_{1,t}| \le (\eta_t/s_t)|\alpha_{1,t}| \le \eta_t$. The same argument gives $|\beta_{1,t+1}| \le \eta_t$. Combining the two cases, $|\alpha_{1,t+1}|, |\beta_{1,t+1}| \le \max\{\eta_t,\ r_t-\eta_t/\sqrt{2}\}$. Next, for $i=2,\dots,r$, \eqref{eq:ss-rec-Newton--Muon} gives $|\beta_{i,t+1}|=\bigl||\beta_{i,t}|-\eta_t\bigr|$. Since $|\beta_{i,t}| \le r_t$, this implies $|\beta_{i,t+1}| \le \max\{\eta_t,\ r_t-\eta_t\} \le \max\{\eta_t,\ r_t-\eta_t/\sqrt{2}\}$. Hence $r_{t+1} \le \max\{\eta_t,\ r_t-\eta_t/\sqrt{2}\}$. The greedy choice minimizes the right-hand side by equalizing the two terms: $\eta_t = r_t-\eta_t/\sqrt{2}$. Solving gives $\eta_t = (2-\sqrt{2})\,r_t$, and therefore $r_{t+1} = (2-\sqrt{2}) r_t$.

For Muon, let $s_t \coloneqq \sqrt{\kappa^2\alpha_{1,t}^2+\beta_{1,t}^2}$. By \eqref{eq:ss-rec-muon}, we have $\alpha_{1,t+1} = (1-\eta_t\kappa/s_t)\alpha_{1,t}$ and $\beta_{1,t+1} = (1-\eta_t/s_t)\beta_{1,t}$. Since $|\alpha_{1,t}|, |\beta_{1,t}| \le r_t$, we have $s_t \le \sqrt{\kappa^2+1}\,r_t$. 
Applying the same argument as above, we have $r_{t+1} \le \max\{\eta_t,\ r_t-\eta_t/\sqrt{\kappa^2+1}\}$. The greedy choice equalizes the two terms, so $\eta_t = r_t-\eta_t/\sqrt{\kappa^2+1}$, and therefore $r_{t+1} = \sqrt{\kappa^2+1}/(\sqrt{\kappa^2+1}+1) r_t$.

Iterating the three recursions shows that gradient descent, Newton--Muon, and Muon all decrease $r_t$ geometrically. Solving $r_t\le \varepsilon$ gives $T_{\mathrm{GD}}(\varepsilon)=O\big(\log(r_0/\varepsilon)/\log((\kappa+1)/(\kappa-1))\big)$, $T_{\mathrm{NM}}(\varepsilon)=O(\log(r_0/\varepsilon))$, and $T_{\mathrm{M}}(\varepsilon)=O\big(\log(r_0/\varepsilon)/\log(1+1/\sqrt{\kappa^2+1})\big)$. Finally, for large $\kappa$,
\[
\log\Big(\frac{\kappa+1}{\kappa-1}\Big)\asymp \frac{1}{\kappa},
\qquad
\log\Big(1+\frac{1}{\sqrt{\kappa^2+1}}\Big)\asymp \frac{1}{\kappa}.
\]
Therefore $T_{\mathrm{GD}}(\varepsilon)=O(\kappa\log(r_0/\varepsilon))$, $T_{\mathrm{NM}}(\varepsilon)=O(\log(r_0/\varepsilon))$, and $T_{\mathrm{M}}(\varepsilon)=O(\kappa\log(r_0/\varepsilon))$.
\end{proof}

\section{One-Step Analysis of Newton--Muon}
\label{sec:numerical-sim}

The convergence analysis in Section~\ref{sec:single-spike-convergence} relies on a simple quadratic case study in order to obtain explicit global dynamics. In this section, we provide a complementary one-step analysis that can be carried out more generally through local approximation.
It has two parts: (i) we introduce a scale-invariant quadratic score that measures the best achievable one-step decrease along a chosen direction after an optimal line search; and (ii) we use a simple spiked activation model to obtain a concrete numerical study. The theoretical derivation for isotropic activations is deferred to Appendix~\ref{app:quadratic-score}.

\subsection{Quadratic Score}

We compare several update directions by their predicted one-step improvement under the surrogate \eqref{eq:triplet-obj-derived}. 
Note first that the objective \eqref{eq:triplet-obj-derived} is a quadratic function of the update
direction ${Q}$, and we only need to choose a direction ${Q}$ and then perform a one-dimensional line search along that direction. This reduces the comparison to a direction choice followed by one-dimensional line search.

Concretely, for a fixed direction ${Q}\in\R^{m\times n}$ and step size $\eta\in\R$, consider the update ${W}\mapsto {W}-\eta {Q}$.
Plugging $\eta {Q}$ into the quadratic surrogate \eqref{eq:triplet-obj-derived} yields the one-dimensional model
\[
  J(\eta {Q})
   = 
  - \eta \mathrm{tr}({Q}{G}^\top)
  + \frac{\eta^2}{2N} \mathrm{tr} \Big({H} {Q} ({Z}{Z}^\top) {Q}^\top\Big).
\]
Under Assumption~\ref{assump:exact-newton}, the optimal step size is $\eta^\star=\mathrm{tr}({Q}{G}^\top)/\mathrm{tr}\big({H}{Q}({Z}{Z}^\top/N){Q}^\top\big)$.
Thus, we define the score of a direction ${Q}$ to be proportional to the best one-step loss decrease under the surrogate
\begin{equation}
  \label{eq:score-matrix-form}
  s({Q})
  \coloneqq
  -2\big(J(\eta^\star {Q})-J(0)\big)
  =
  \frac{\mathrm{tr}({Q}{G}^\top)^2}{\mathrm{tr}\Big({H}{Q}({Z}{Z}^\top/N){Q}^\top\Big)}.
\end{equation}
This provides a general, scale-invariant criterion for comparing different update directions.

\subsection{Numerical Study with Spiked Activation}
We numerically evaluate the score $s({Q})$ in \eqref{eq:score-matrix-form} for the six directions under ${G}={H}({W}-W^\star)({Z}{Z}^\top/N)$. 
These directions are 
(i) ${Q}={G}$,
(ii) ${Q}_{\mathrm{Muon\text{-}SVD}}=\mathrm{msgn}({G})$ (exact SVD),
(iii) ${Q}_{\mathrm{Muon\text{-}NS}}$, obtained by applying five Newton--Schulz iterations to ${G}$, specifically, ${Q}_{\mathrm{Muon\text{-}NS}}={X}_5$, where ${X}_0={G}/\|{G}\|_F$ and ${X}_{t+1}=\big(3.4445{I}_m-4.7750{X}_t{X}_t^\top+2.0315({X}_t{X}_t^\top)^2\big){X}_t$ for $t=0,\dots,4$,
(iv) ${Q}_{\mathrm{Newton\text{-}Muon\text{-}SVD}}=\mathrm{msgn}\big({G}({Z}{Z}^\top)^{-1}\big)$ (exact SVD),
(v) ${Q}_{\mathrm{Newton\text{-}Muon\text{-}NS}}$, obtained by applying the same five Newton--Schulz iterations to ${G}({Z}{Z}^\top)^{-1}$,
and (vi) ${Q}_{\mathrm{Newton}}={H}^{-1}{G}({Z}{Z}^\top)^{-1}$, which is proportional to ${W}-W^\star$.
Throughout, we fix the square setting $m=n=512$, $\lambda_{\max}=1$, and $\lambda_{\min}=10^{-4}$. We generate the eigenvalues of ${H}$ by the stretched-exponential rule
\[
\lambda_k=\lambda_{\max}\exp\big(-\tau (k-1)^p\big),\qquad
\tau=\frac{\log(\lambda_{\max}/\lambda_{\min})}{(m-1)^p},
\]
so that $\lambda_m=\lambda_{\min}$, then sample a random orthogonal ${P}$ and set ${H}={P} \mathrm{diag}(\lambda_1,\ldots,\lambda_m) {P}^\top$. We sample ${W}-W^\star\in\R^{m\times n}$ with i.i.d.\ $\mathcal{N}(0,1)$ entries. For the activation matrix, we sample the columns $\boldsymbol{z}_i\in\R^n$ of ${Z}$ independently from $\mathcal{N}(0,\mathrm{diag}(\kappa,1,\ldots,1))$ with $\kappa=64$, so that the population activation second moment is spiked. We use three choices of $(N,p)$: the baseline case $(8192,0.3)$, a more top-uniform curvature case $(8192,2.4)$, and a smaller-sample case $(1024,0.3)$. For each setting, we run $1024$ independent simulations and report the mean score with the 2.5\%--97.5\% interval.

Figure~\ref{fig:sim-spiked-activation} shows substantially higher scores for Newton--Muon than for Muon when both activation anisotropy and curvature anisotropy are strong, with Newton--Muon being closer to the optimal Newton direction. Muon also substantially outperforms gradient descent when the curvature is anisotropic.

\begin{figure}[htb]
  \centering
  \includegraphics[width=\linewidth]{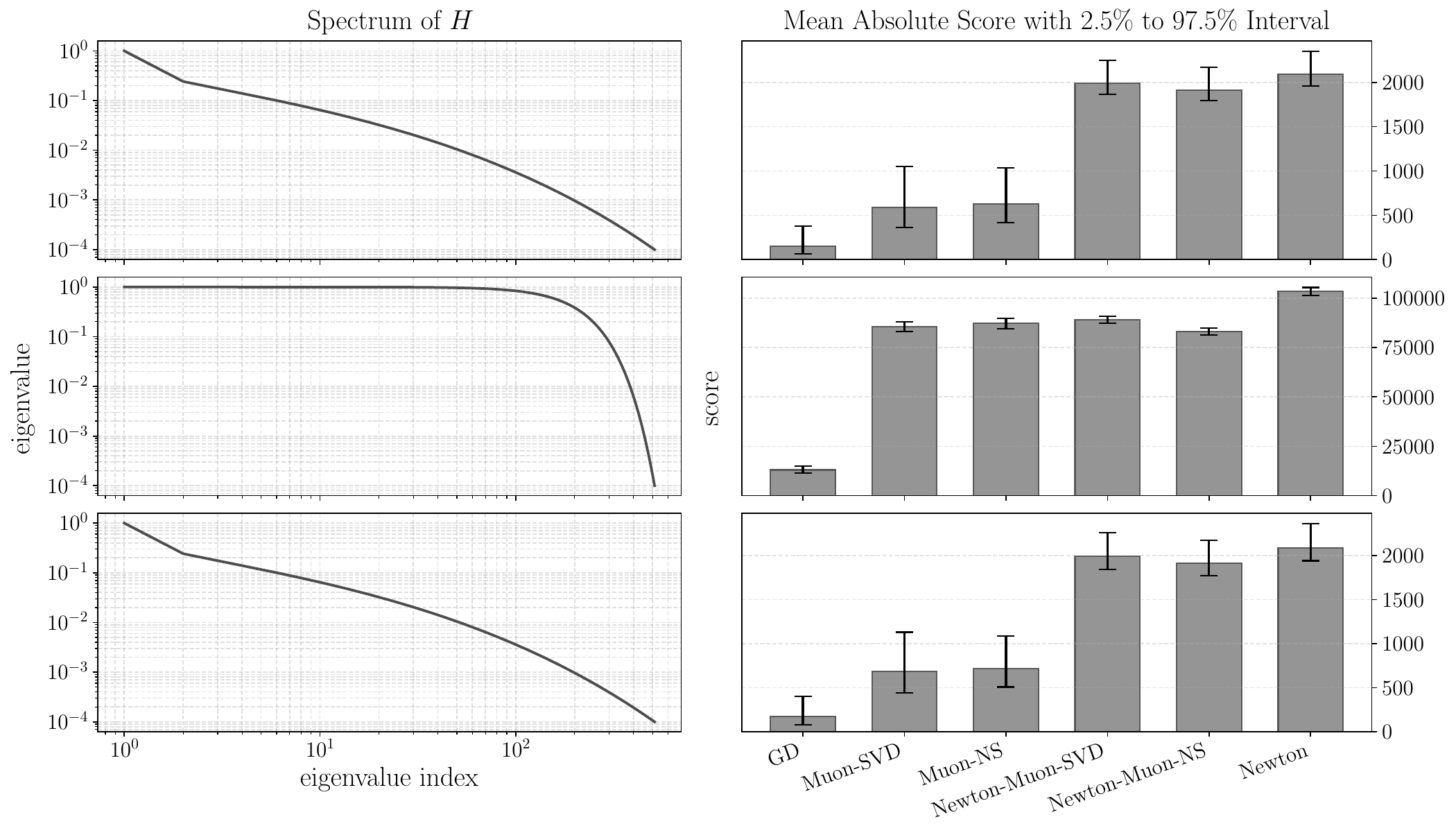}
  \caption{Numerical study with spiked activation second moment $\mathrm{diag}(\kappa,1,\ldots,1)$ and $\kappa=64$. Top: baseline case $(N,p)=(8192,0.3)$. Middle: more uniform curvature $(N,p)=(8192,2.4)$. Bottom: smaller sample size $(N,p)=(1024,0.3)$. The left column shows the spectrum of ${H}$, and the right column shows the corresponding mean absolute scores $s({Q})$.}
  \label{fig:sim-spiked-activation}
\end{figure}

\section{LLM Experiments}
\label{sec:real-llm-exp}

\subsection{Pretraining Benchmark Records on Modded-NanoGPT}
\label{subsec:llm-exp}

\paragraph{Benchmark setup.} We first compare Newton--Muon, Muon, and AdamW using historical records from the Modded-NanoGPT speedrun benchmark~\citep{modded_nanogpt_2024}. This benchmark is designed around a fixed objective: train a GPT-style language model on FineWeb~\citep{penedo2024the} and report the wall-clock time needed to reach a target validation loss under a fixed hardware configuration. The benchmark has two tracks: the short track targets a validation loss of 3.28, while the medium track lowers the target to 2.92. For reproducibility, several important implementation details are presented in Appendix~\ref{app:llm-details}.

\paragraph{Reference records.} In this benchmark, each Record~\#k refers to a historical leaderboard submission together with its released training configuration and runtime log. In our experiments, we use these records for reproduction and comparison. We use Record~\#4 from the short track as a baseline; this run was submitted shortly after Muon was introduced. To benchmark
against a stronger setup, for the short track we adapt a training script that is close in configuration to
Record~\#28. For the medium track, we compare against Record~\#17.

However, note that the later records are heavily optimized around the original Muon update, which makes them difficult to surpass without extensive hyperparameter tuning.
Since we only tune the learning rate and Newton--Muon hyperparameters, the baseline comparison should be viewed as the most direct comparison of Newton--Muon, while other comparisons should be interpreted as evaluations under training configurations that were heavily tuned for the original Muon update rather than for Newton--Muon.
Algorithm~\ref{alg:newton-muon} summarizes the implementation of Newton--Muon.

\begin{algorithm}[htb]
\caption{Newton--Muon}
\label{alg:newton-muon}
\DontPrintSemicolon

\KwIn{For each Muon layer $\ell$: input activations ${Z}_\ell\in\mathbb{R}^{n\times N}$, gradient ${G}_\ell\in\mathbb{R}^{m\times n}$, running second moment ${K}_\ell\in\mathbb{R}^{n\times n}$ (initialize with $10^{-3}{I}_n$), stored inverse ${K}_\ell^{-1}$.}
\KwIn{Hyperparameters: EWMA coefficient $\beta$, ridge scaling $\gamma$, refresh interval $k$.}

\BlankLine
\For{\texttt{step}$=0,1,\dots$}{
    \ForEach{$\ell$}{
        \If{$(\texttt{step}+1)\bmod k = 0$}{
            ${K}_\ell \gets \beta {K}_\ell + (1-\beta) {Z}_\ell {Z}_\ell^\top / N$\;
            \tcp*[r]{Compute via a symmetric rank-$N$ update (Appendix~\ref{app:symm})}
            $\gamma_\ell \gets \gamma\cdot \mathrm{tr}({K}_\ell)/n$\;
            ${K}_\ell^{-1} \gets ({K}_\ell + \gamma_\ell {I}_n)^{-1}$\;
            \tcp*[r]{Compute via Cholesky inverse (Appendix~\ref{app:chol-inverse}) or a polynomial iteration with a custom symmetric-output matmul (Appendix~\ref{app:poly-inv})}
        }
        ${G}_\ell \gets {G}_\ell {K}_\ell^{-1}$\;
        \tcp*[r]{${K}_\ell^{-1}$ is applied to the raw layer gradient, before momentum, weight decay, and the remaining Muon pipeline}
        Apply the standard Muon update using the right-preconditioned gradient ${G}_\ell$\;
    }
}
\end{algorithm}

\paragraph{Baseline.}
We begin with the short track Record~\#4 with a single NVIDIA H100 GPU. This run trains a 124M GPT-2 architecture with 3.25B training tokens. The Muon baseline
learning rate is $0.0036$. We modified the code to run on a single GPU without touching any training pipeline. For Newton--Muon, we use learning
rate $0.0040$, EWMA coefficient $\beta=0.95$, ridge scaling $\gamma=0.2$, and refresh interval $k=32$. For AdamW, we assign the transformer block parameters a reduced AdamW learning rate $0.000576$, selected from a learning-rate sweep. All wall-clock numbers reported below are measured in our environment under this configuration. 
Table~\ref{tab:record4-summary} summarizes the final validation losses and total training times. Figure~\ref{fig:record4-baseline-plot} compares the validation loss trajectories.

\begin{table}[htb]
  \centering
  \begin{tabular}{l r r}
    \hline
    Method & Loss & Time (s) \\
    \hline
    AdamW & 3.3801 & 7228.4 \\
    Muon & 3.2793 & 7314.1 \\
    Newton--Muon & 3.2611 & 7443.3 \\
    \hline
  \end{tabular}
  \caption{Short track Record~\#4 setting (single H100).}
  \label{tab:record4-summary}
\end{table}

\paragraph{Ablation for baseline.}
Unless otherwise stated, we fix the Newton--Muon settings to learning rate $0.0040$, EWMA coefficient $\beta=0.95$, ridge scaling
$\gamma=0.2$, and refresh interval $k=32$, and vary only the ablation parameters. Figure~\ref{fig:ablation-refresh}
shows a two-dimensional sweep over refresh interval $k$ and EWMA coefficient $\beta$.
Figure~\ref{fig:ablation-ridge-lr} sweeps ridge scaling $\gamma$ and the learning rate.

\begin{figure}[htb]
  \centering
  \includegraphics[width=0.9\linewidth]{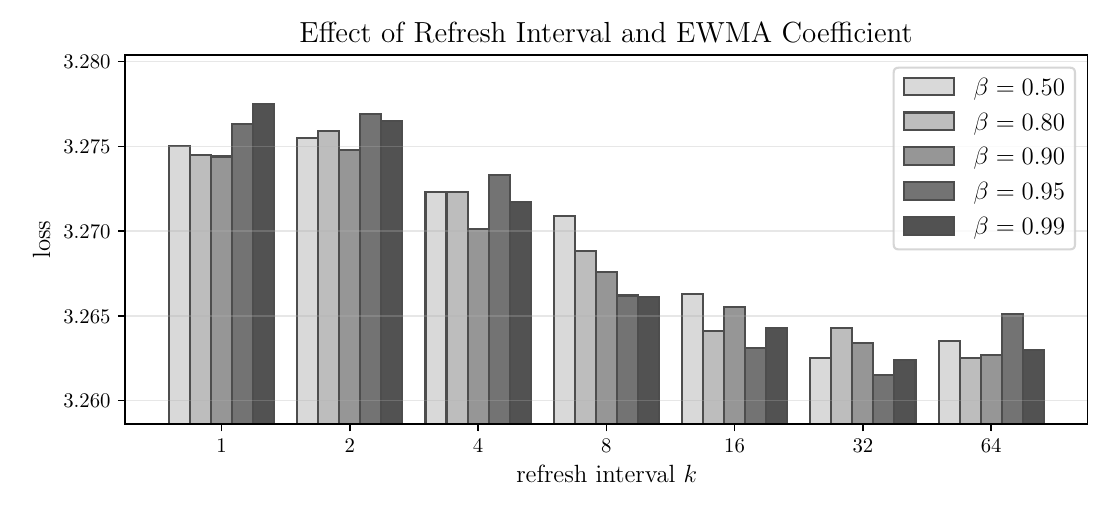}
  \caption{Refresh ablation for Newton--Muon on short track Record~\#4. Grouped bar plot over refresh interval $k$, with one bar
  per EWMA coefficient $\beta$ (ridge scaling fixed at $\gamma=0.2$ and learning rate fixed at $0.0040$).}
  \label{fig:ablation-refresh}
\end{figure}

\begin{figure}[htb]
  \centering
  \includegraphics[width=1\linewidth]{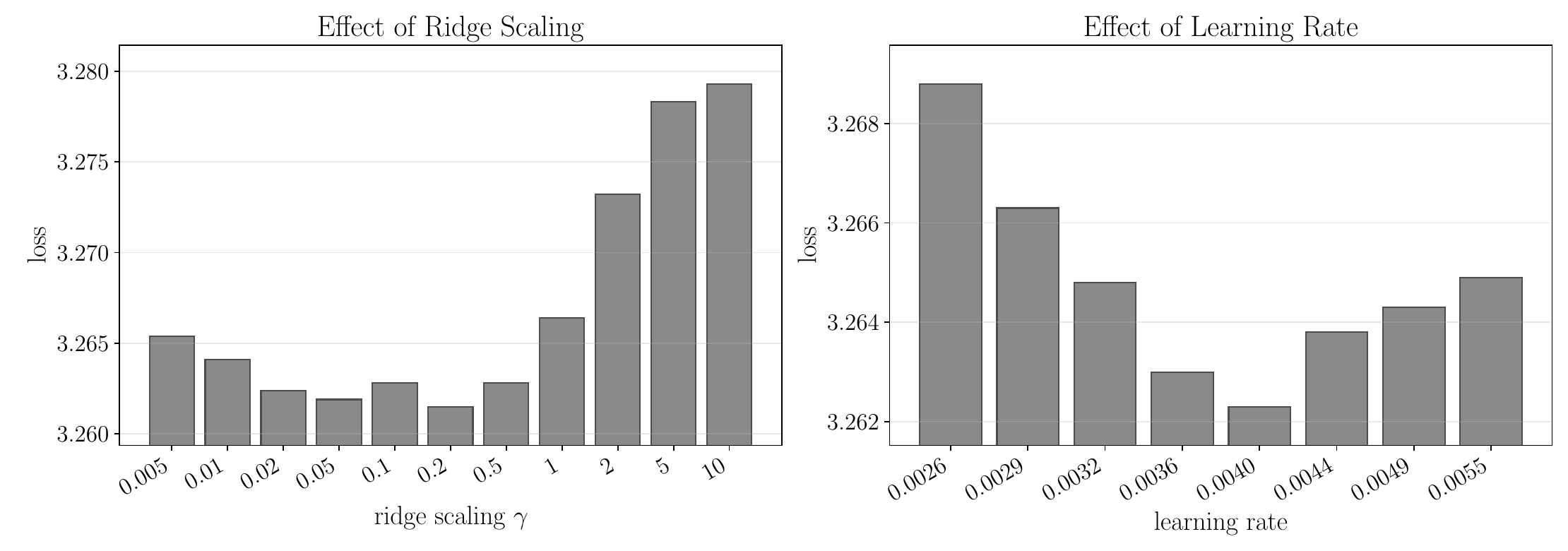}
  \caption{Left: Ridge-scaling ablation for Newton--Muon on short track Record~\#4. Bar plot over ridge scaling $\gamma$ with
  $k=32$, $\beta=0.95$, and learning rate $0.0040$ fixed. Right: Learning-rate ablation for Newton--Muon on short track Record~\#4. Bar plot over learning rate with
  $k=32$, $\beta=0.95$, and $\gamma=0.2$ fixed.}
  \label{fig:ablation-ridge-lr}
\end{figure}

These ablations suggest three patterns. First, Figure~\ref{fig:ablation-refresh} shows that very frequent refreshes ($k\in\{1,2\}$) consistently underperform, while
moderate refresh intervals ($k\in\{16,32,64\}$) yield the best final losses. 
Second, ridge scaling is essential to make the right preconditioner numerically well behaved. Figure~\ref{fig:ablation-ridge-lr} indicates that overly small ridge values can degrade
training slightly since the right-preconditioner may be ill-conditioned. In contrast, a wide mid-range (roughly
$\gamma\in[0.02,0.5]$ here) performs similarly. Very large ridge values eventually move the update back toward
standard Muon, which corresponds to the loss rising toward the baseline 3.2793 as $\gamma$ increases.
Third, the learning-rate sweep in Figure~\ref{fig:ablation-ridge-lr} is relatively flat around the best region, and the best Newton--Muon learning rate is very close to the best Muon learning rate.

\paragraph{Short Track Record~\#28.}
We reproduce a short track configuration similar to a record submitted around the same time as Record~\#28 on a single NVIDIA L40S GPU. This run trains a $\sim$275M-parameter model with about $670$M training tokens. We use the training script at
\url{https://github.com/KellerJordan/modded-nanogpt/blob/9d9dc969c451c87b7ad3c84f807db2c2d9109f41/train_gpt.py}. For Newton--Muon, we do not change any other configurations including the learning rate. We use EWMA coefficient $\beta=0.8$, ridge scaling $\gamma=0.2$, and refresh interval $k=16$; since
the run is shorter than Record~\#4, we refresh more aggressively. 
For AdamW, we assign the transformer block matrices a reduced AdamW learning rate $0.00055$.
Because this setting is noisy, we run experiments four times and summarize results in Table~\ref{tab:record28-summary}. Figure~\ref{fig:record28-plot} compares the validation-loss curves of
the second-best run (by final validation loss) from each group. Although Newton--Muon has slightly higher total runtime due to preconditioning overhead (Table~\ref{tab:record28-summary}), its loss-versus-time curve indicates a small advantage in time to reach comparable validation loss (Figure~\ref{fig:record28-plot}).

\begin{table}[htb]
  \centering
  \begin{tabular}{l r r r r r r}
    \hline
    Method & Run 1 & Run 2 & Run 3 & Run 4 & Avg.\ loss & Avg.\ time (s) \\
    \hline
    AdamW & 3.4677 & 3.4631 & 3.4640 & 3.4566 & 3.4628 & 4272.0 \\
    Muon & 3.2758 & 3.2777 & 3.2783 & 3.2830 & 3.2787 & 4305.9 \\
    Newton--Muon & 3.2733 & 3.2736 & 3.2740 & 3.2745 & 3.2739 & 4342.4 \\
    \hline
  \end{tabular}
  \caption{Short track Record~\#28 setting (single L40S).}
  \label{tab:record28-summary}
\end{table}

\begin{figure}[htb]
  \centering
  \includegraphics[width=\linewidth]{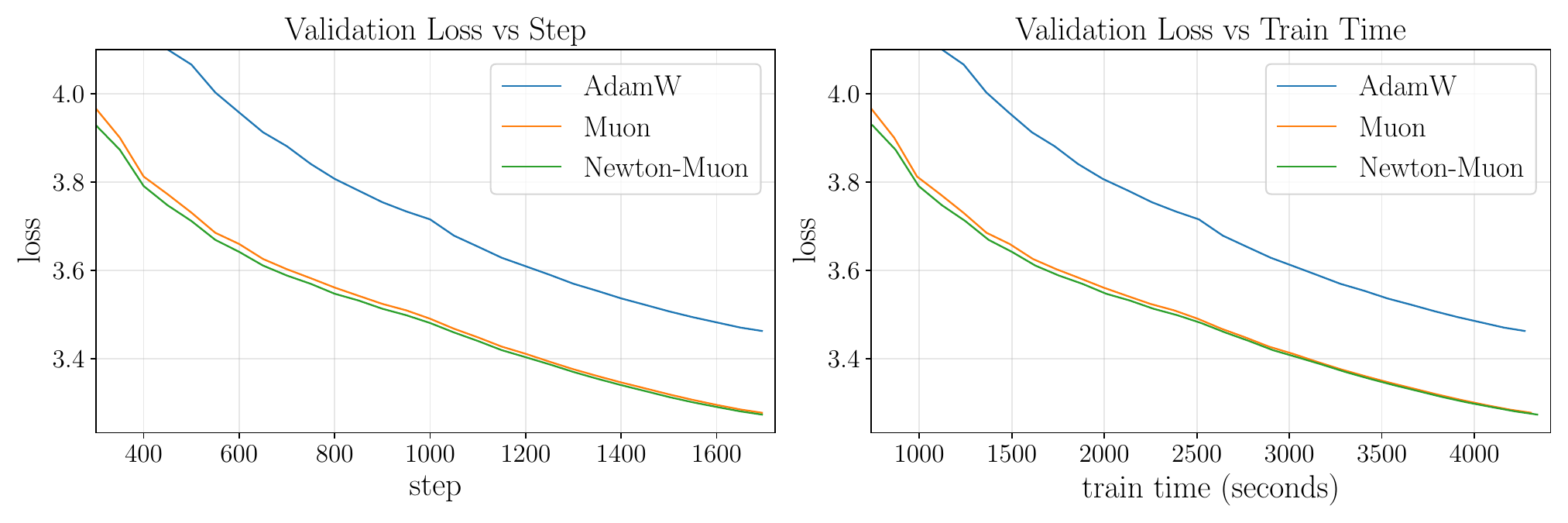}
  \caption{Validation loss trajectories for the short track Record~\#28 setting (single L40S). We plot the second-best run for AdamW, Muon, and Newton--Muon, and show loss versus step (left) and loss versus wall-clock time (right).}
  \label{fig:record28-plot}
\end{figure}

\paragraph{Medium Track Record~\#17.}
We next reproduce the medium track Record~\#17 setting. This run trains a $\sim$455M-parameter model with about $3.12$B training tokens. Similarly, for Newton--Muon we keep other configurations fixed. We use EWMA coefficient $\beta=0.9$, ridge scaling $\gamma=0.2$, and refresh interval $k=24$. We run both Muon and Newton--Muon three times and summarize results in Table~\ref{tab:record17-summary}. AdamW is not included for the medium track comparison. The improvement in final validation loss is only marginal in this setting.

\begin{table}[htb]
  \centering
  \begin{tabular}{l r r r r}
    \hline
    Method & Run 1 & Run 2 & Run 3 & Avg.\ loss \\
    \hline
    Muon   & 2.9190 & 2.9208 & 2.9190 & 2.9196 \\
    Newton--Muon & 2.9175 & 2.9191 & 2.9181 & 2.9183 \\
    \hline
  \end{tabular}
    \caption{Medium track Record~\#17 setting. Wall-clock time is not reported since we did not observe a substantial improvement.}
  \label{tab:record17-summary}
\end{table}

\subsection{Quadratic Score}
\label{subsec:real-score}
We evaluate candidate directions using the same one-dimensional surrogate, but with the exact parameter-space Hessian in the denominator. Given a parameter matrix ${W}$ and a candidate direction ${Q}$, we define the score
\begin{equation}
\label{eq:score-quad}
s({Q}) = \frac{\mathrm{tr}({G}^\top {Q})^2}{\mathrm{vec}({Q})^\top {\mathcal{H}}_{{W}} \mathrm{vec}({Q})}.
\end{equation}
When $\mathrm{vec}({Q})^\top {\mathcal{H}}_{{W}} \mathrm{vec}({Q})>0$, this quantity is, up to a constant factor, the predicted best decrease of the local quadratic model after an optimal line search along ${Q}$.

We compute the curvature term $\mathrm{vec}({Q})^\top {\mathcal{H}}_{{W}} \mathrm{vec}({Q})$ using exact Hessian-vector products for the batch loss, without any Kronecker approximation.
To make scores comparable across methods and layers, we normalize each layer's raw gradient to unit Frobenius norm
$\widehat{{G}} \coloneqq {G} / \|{G}\|_F$.
Each candidate direction is also normalized
$\widehat{{Q}} \coloneqq {Q} / \|{Q}\|_F$.
We then report three quantities. First, the alignment term
$\mathrm{tr}(\widehat{{G}}^\top \widehat{{Q}})$.
Second, the curvature term
$(1/2)\mathrm{vec}(\widehat{{Q}})^\top {\mathcal{H}}_{{W}} \mathrm{vec}(\widehat{{Q}})$.
Third, the resulting score
$s(\widehat{{Q}})$.

\paragraph{Experimental setup.}
We use Pythia-70M~\citep{biderman2023pythia} and evaluate two checkpoints, step 1000 and step 50000. We stream the Pile~\citep{gao2020pile} and form a large batch of $2048$ sequences of length $1024$, totaling $1024 \times 16 \times 128 = 2{,}097{,}152$ tokens. We study four weight matrices in the fourth transformer block: the attention output projection, the attention QKV projection, the MLP expansion, and the MLP contraction. To stabilize computation, we replace ${Z}{Z}^\top$ by the damped matrix ${K}_\gamma \coloneqq {Z}{Z}^\top + \gamma {I}_n$ with $\gamma > 0$, starting from a small damping value and multiplying $\gamma$ by $10$ until ${K}_\gamma$ admits a stable \texttt{float64} Cholesky factorization.

We compare five directions per layer. The gradient descent direction is ${Q} = \widehat{{G}}$. Muon-NS5 and Muon-NS32 apply 5 and 32 Newton--Schulz iterations to $\widehat{{G}}$, respectively. Newton--Muon-NS5 and Newton--Muon-NS32 apply 5 and 32 Newton--Schulz steps to the right-preconditioned gradient.

\paragraph{Results.}
\label{subsec:real-results}

Figure~\ref{fig:real-quadratic-score-combined} summarizes the results. For these four matrices at the two sampled checkpoints, the ranking is consistent. Newton--Muon achieves the highest score, Muon is intermediate, and the raw gradient is worst. By construction, the gradient direction has the largest alignment with itself, but it also typically has a large curvature term. Both Muon and Newton--Muon substantially reduce this curvature term. Relative to Muon, Newton--Muon usually reduces the curvature term more than it reduces the alignment term, which explains its higher score. The advantage of Newton--Muon over Muon is smaller at step 50000 than at step 1000, and this trend appears across all four matrices. Also, NS5 and NS32 are close, suggesting that a small number of Newton--Schulz steps often suffices.

Table~\ref{tab:real-K-stats} reports diagnostics of the input activation second moment computed on the same batch. The matrices are strongly anisotropic, with large diagonal spread, substantial off-diagonal mass, and very large condition numbers.

\begin{table}[htb]
  \centering
  \small
  \setlength{\tabcolsep}{4pt}
  \renewcommand{\arraystretch}{1.12}
  \begin{tabular}{l r r r r r}
    \hline
    Module & $n$ &
    $\kappa({Z}{Z}^\top)$ &
    $d_{\max}/d_{\min}$ &
    $\bar{o}/d_{\mathrm{mean}}$ &
    $\gamma/d_{\mathrm{mean}}$ \\
    \hline
    Attn out   &  512 & $9.18\times 10^{3}$ & $2.84\times 10^{1}$ & $9.89\times 10^{1}$ & $1.0\times 10^{-6}$ \\
    Attn QKV   &  512 & $2.04\times 10^{6}$ & $3.00\times 10^{0}$ & $6.59\times 10^{1}$ & $1.0\times 10^{-6}$ \\
    MLP expansion &  512 & $\infty$            & $3.01\times 10^{0}$ & $6.56\times 10^{1}$ & $1.0\times 10^{-4}$ \\
    MLP contraction  & 2048 & --                  & $7.23\times 10^{1}$ & $1.88\times 10^{2}$ & $1.0\times 10^{-6}$ \\
    \hline
    Attn out   &  512 & $2.34\times 10^{3}$ & $3.02\times 10^{1}$ & $7.21\times 10^{1}$ & $1.0\times 10^{-6}$ \\
    Attn QKV   &  512 & $9.77\times 10^{4}$ & $1.66\times 10^{2}$ & $4.22\times 10^{1}$ & $1.0\times 10^{-6}$ \\
    MLP expansion &  512 & $8.22\times 10^{6}$ & $1.21\times 10^{2}$ & $3.97\times 10^{1}$ & $1.0\times 10^{-6}$ \\
    MLP contraction  & 2048 & --                  & $7.46\times 10^{2}$ & $1.88\times 10^{2}$ & $1.0\times 10^{-6}$ \\
    \hline
  \end{tabular}
  \caption{Diagnostics of the activation second moment ${Z}{Z}^\top$ for four modules at step 1000 (top block) and step 50000 (bottom block). Here $d_{\min}$, $d_{\mathrm{mean}}$, and $d_{\max}$ are the minimum, mean, and maximum diagonal entries of ${Z}{Z}^\top$, $\bar{o}=(1/n)\sum_i\sum_{j\neq i}|({Z}{Z}^\top)_{ij}|$, and $\kappa({Z}{Z}^\top)$ is the condition number of ${Z}{Z}^\top$. The value $\kappa({Z}{Z}^\top)$ is reported when eigenvalues were explicitly computed; ``$\infty$'' indicates numerical instability in the eigenspectrum estimate (a tiny negative eigenvalue at step 1000), and ``--'' indicates that the spectrum was not explicitly computed at $n=2048$. The last column reports the relative damping level selected by the adaptive Cholesky procedure.}
  \label{tab:real-K-stats}
\end{table}

\begin{figure}[htbp]
  \centering

  \begin{minipage}{1.00\linewidth}
    \centering
    \includegraphics[width=\linewidth]{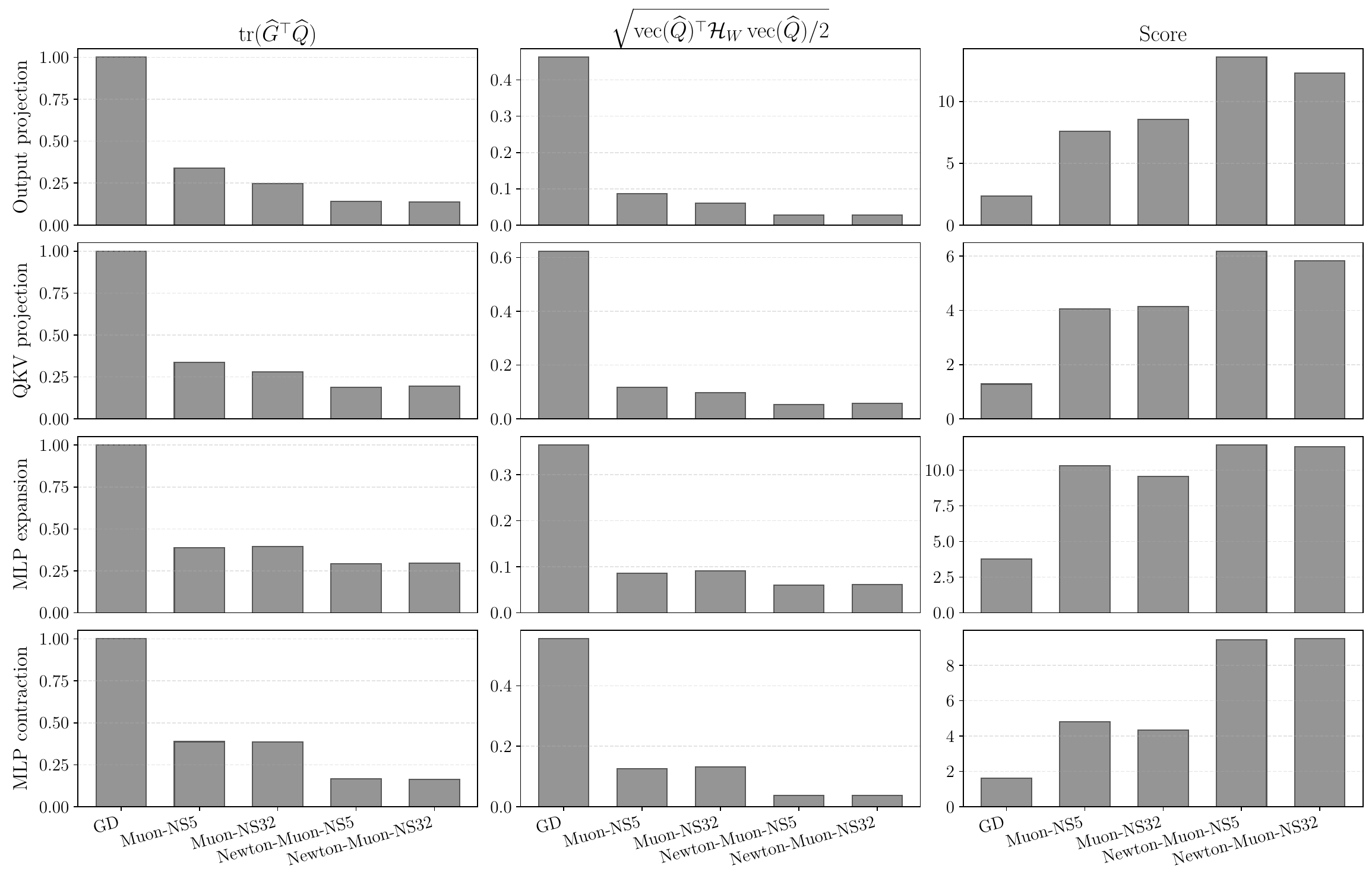}
  \end{minipage}

  \vspace{0.8em}

  \begin{minipage}{1.00\linewidth}
    \centering
    \includegraphics[width=\linewidth]{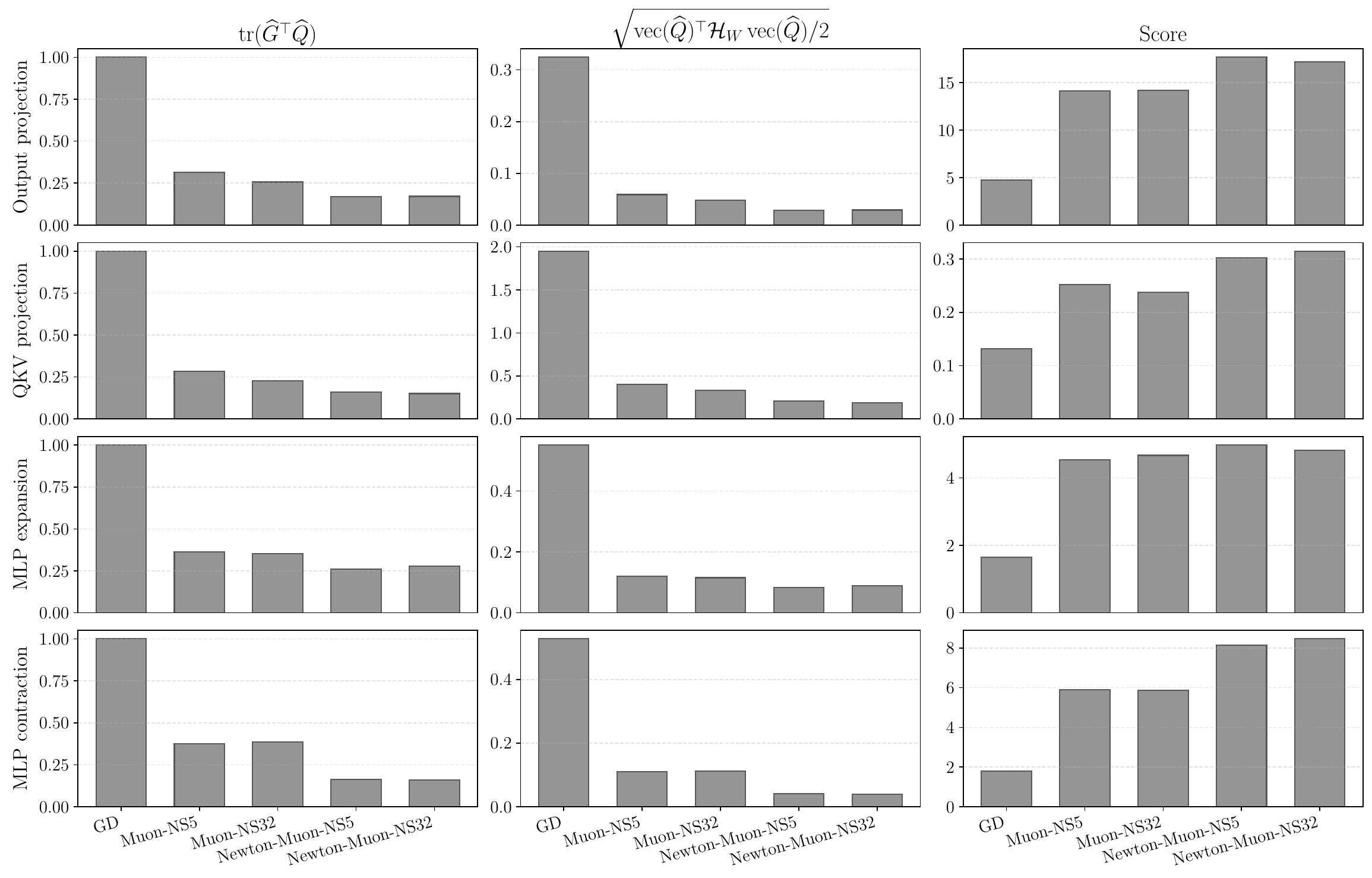}
  \end{minipage}

  \caption{Quadratic score comparison on four layer matrices from Pythia at checkpoints step 1000 (top) and step 50000 (bottom), corresponding to the one-step analysis in~\eqref{eq:score-quad}.}
  \label{fig:real-quadratic-score-combined}
\end{figure}

\section{Discussion}
\label{sec:conclusion}

In this work, we introduce a triplet quadratic surrogate model that offers a local second-order view of Muon and leads to a new optimizer, Newton--Muon. Under an isotropic proxy, one-step minimization yields the update $\mathrm{msgn}(G(ZZ^\top)^{-1})$, suggesting that Muon is an implicit Newton method without right preconditioning by the input second moment. Empirically, on our reproductions of historical Modded-NanoGPT speedrun benchmark configurations, Newton--Muon reaches the target validation loss in 6\% fewer steps and reduces wall-clock time to that loss by over 4\% relative to the Muon baseline. However, several limitations of the current framework remain and point to important directions for future work.

First, Newton--Muon relies on the isotropic proxy ${\Sigma}_W \propto {I}_m$. This choice is simple and robust, but it discards potentially useful information about the displacement distribution and may become inaccurate later in training. It is therefore better viewed as a practical proxy than as a faithful model throughout optimization. Indeed, weight matrices in classification models often exhibit low dimensional geometric structure, as suggested by neural collapse~\citep{papyan2020prevalence} and minority collapse~\citep{fang2021exploring}; related phenomena, including linguistic collapse~\citep{wu2024linguistic} and cluster formation in attention dynamics~\citep{geshkovski2023emergence}, have also been observed in transformers. A natural direction for future work is to estimate ${\Sigma}_{{W}}$ from training dynamics. With a sufficiently accurate estimate of ${\Sigma}_{{W}}$, one need not explicitly compute its SVD or symmetric square root in order to apply \eqref{eq:SigmaW-polar-update3}; it suffices to obtain a factorization ${\Sigma}_{{W}} = {M}{M}^\top$, and then use ${M}$ to implement the same update direction. More details on this factorized implementation are deferred to Appendix~\ref{app:llm-exp-sigmaW}. However, our preliminary experiments in Appendix~\ref{app:llm-exp-sigmaW} with non-identity proxies, such as diagonal and factorized forms, did not improve performance, suggesting that estimating ${\Sigma}_{{W}}$ reliably is nontrivial and may introduce feedback bias. Developing stable, cheap estimators for non-isotropic proxies remains an important open problem.

Second, the Kronecker approximation of the parameter-space Hessian omits explicit token coupling in transformer attention mechanisms. In general, for token indices $t \neq s$, one can have $H_{ts} = \partial^2 L / (\partial \boldsymbol{y}_t \partial \boldsymbol{y}_s) \neq 0$, but these cross-token curvature terms are discarded when passing from the exact token-coupled formula~\eqref{eq:HW-transformer-exact} to the averaged Kronecker approximation~\eqref{eq:HW-transformer-kron}. We also use a single shared curvature matrix $H$ across all samples and positions, whereas the local output space curvature may vary with sample or token. As a result, the approximation $\mathcal{H}_W \approx (ZZ^\top/N)\otimes H$ may become less accurate in late training, when the Hessian can deviate further from Kronecker structure. These may help explain why Muon's advantage decays in very late training~\citep{wen2025fantastic}. Future work should develop tractable approximations that incorporate token-coupled curvature and relax the shared $H$ assumption.

Third, our experiments computed the activation second moment inverse using a damped Cholesky solve of ${K}_\gamma = {Z}{Z}^\top + \gamma I_n$. In most cases we used the full inverse of ${K}_\gamma$, but for MLP contraction matrices of shape $d \times (4d)$ we instead used a block-diagonal approximation with four $d \times d$ blocks, which performed well in practice (see Appendix~\ref{app:llm-details}). This suggests that structured approximations to ${K}_\gamma$ may already suffice. Therefore, future work could study structured approximations such as block matrices or low-rank factorizations, which may substantially reduce memory and computation while preserving most of the benefit of Newton--Muon. Future work could also compute the activation second-moment inverse via polynomial iteration (Appendix~\ref{app:poly-inv}), which may be faster. It is also natural to replace our heuristic damping $\gamma \propto \mathrm{tr}(ZZ^\top)/n$ by a shifted Cholesky rule~\citep{fukaya2020shifted}, for example $\gamma \propto \|ZZ^\top\|_2$ or $\gamma \propto \|ZZ^\top\|_F/\sqrt{n}$. 

Fourth, future work should evaluate Newton--Muon in distributed training, since we only test it on a single GPU. In particular, it remains important to understand how to compute and invert the activation second moment ${Z}{Z}^\top$ in multi-GPU training. This includes both efficient implementations and approximations that preserve most of the benefit of Newton--Muon. Studying these systems issues at scale is an important direction for future work.

\subsection*{Acknowledgments}
This research was supported in part by the Wharton AI fund.

\bibliographystyle{plainnat}
\bibliography{ref}

\clearpage
\appendix 
\appendixpage
\addappheadtotoc 

\section{LLM Experimental Details}
\label{app:llm-details}
This appendix collects the implementation and benchmarking details for the LLM experiments in Section~\ref{subsec:llm-exp}, including several details that are important for reproducing the reported results.

\paragraph{Comparison protocol.}
For a fair comparison, we use several publicly logged benchmark records as configuration references and report wall-clock results from our reproductions. To keep the comparison controlled, for each selected record we modify only the optimizer and tune learning rates when needed. For Newton--Muon, we replace the Muon update with Newton--Muon. For AdamW, to keep the comparison controlled, we replace the Muon-optimized parameters by AdamW, set $\beta=(0.9,0.999)$, and tune the AdamW learning rate over a wide range, reporting the setting that achieves the lowest validation loss, while keeping the remaining training settings unchanged. All other training settings are kept identical to the original record; we may make implementation-level changes (e.g., custom kernels) only to improve efficiency without altering the underlying training pipeline, and we apply any such changes uniformly across methods within the same comparison. For each selected record, we reproduce the run in our own environment and report wall-clock time from these reproductions. We do not claim that our measured times match the leaderboard submission times; our focus is the within-environment wall-clock difference among Newton--Muon, Muon, and AdamW under otherwise identical training settings.

\paragraph{Layerwise second moment computation.}
For the attention QKV projection and the MLP expansion, the input dimension is \(d\), so we form a single \(d\times d\)
activation second moment matrix from their inputs. For the attention output projection, the input dimension is also \(d\),
but the activations are taken from the attention output immediately before applying the output projection.
For the MLP contraction, the input dimension is \(4d\). To avoid forming a full \((4d)\times(4d)\) second moment, we use a
block-diagonal approximation: reshape each activation \(\boldsymbol{z}\in\mathbb{R}^{4d}\) into four
contiguous blocks \(\boldsymbol{z}=[\boldsymbol{z}^{(1)};\ldots;\boldsymbol{z}^{(4)}]\) with
\(\boldsymbol{z}^{(b)}\in\mathbb{R}^{d}\), and form four separate \(d\times d\) second moments
\({Z}^{(b)}{{Z}^{(b)}}^\top\). The corresponding inverse is represented as four independent \(d\times d\)
inverses, and applying it is done by splitting the MLP contraction gradient into four \(d\times d\) blocks along the input
dimension and right-multiplying each block by its matching inverse.

\paragraph{Numerical precision.}
There are also a few important implementation details regarding numerical precision. First, the activation second moment update
${Z}_\ell{Z}_\ell^\top$ is relatively well behaved, and in practice the accumulation can be done with \texttt{bfloat16}. However, the inverse computation is much more sensitive. We compute 
${K}_\ell^{-1}=({K}_\ell+\gamma_\ell{I}_n)^{-1}$ using a \texttt{float32} Cholesky factorization followed by a Cholesky inverse (see Appendix~\ref{app:chol-inverse}). 
Second, when applying the right-preconditioner, the matrix multiply ${G}_\ell {K}_\ell^{-1}$ should also be performed in
\texttt{float32}. If the gradients are stored in \texttt{bfloat16}, we can upcast ${G}_\ell$ to \texttt{float32} before the multiplication.

\section{Optimizing Computation}
\label{app:opt-compute}

\subsection{Symmetric Matrix Multiplication}
\label{app:symm}

Our implementation exploits the fact that the activation second moment
\({Z}{Z}^\top\) (and any polynomial \(p({Z}{Z}^\top)\)) is symmetric. This
structure allows us to avoid redundant computation by computing only one triangle and reconstructing the other by
symmetry.

\paragraph{SYRK (symmetric rank-\(k\) update).}
Forming the second moment
\[
{K}  \gets  {Z}{Z}^\top
\]
is exactly a SYRK pattern: it suffices to compute either the lower or upper triangle of \({K}\) and fill the
other half by mirroring across the diagonal.

\paragraph{SYPP (Symmetric polynomial product).}
In the polynomial inverse, we repeatedly multiply factors that are polynomials in \({Z}{Z}^\top\).
Since \({Z}{Z}^\top\) is symmetric, any polynomial \(p({Z}{Z}^\top)\) is also
symmetric. Moreover, because these factors are functions of the same matrix, they commute, so products such as
\(p({Z}{Z}^\top) q({Z}{Z}^\top)\) remain symmetric. Consequently, we can compute
only one triangle of the product and write the mirrored entries to complete the matrix, reducing nearly half of the work
compared to a dense matrix multiplication.

PyTorch~\citep{paszke2019pytorch} does not expose a simple interface for enforcing this kind of triangular compute. Instead, we implement Triton~\citep{tillet2019triton} custom kernels, adapted from
the Modded-NanoGPT repository~\citep{modded_nanogpt_2024}.

\subsection{Cholesky Inverse}
\label{app:chol-inverse}

When ${K}_\gamma={Z}{Z}^\top+\gamma{I}_n\succ 0$, Cholesky factorization gives
${K}_\gamma={L}{L}^\top$ with ${L}$ lower triangular. We then explicitly form the
inverse ${K}_\gamma^{-1}=({L}{L}^\top)^{-1}={L}^{-\top}{L}^{-1}$.
We use \texttt{torch.linalg.cholesky\_ex} followed by \texttt{torch.cholesky\_inverse} in \texttt{float32}.

\subsection{Polynomial Iteration Inverse}
\label{app:poly-inv}

We also consider a polynomial iteration that explicitly constructs an approximation to
${K}_\gamma^{-1}$ using only SYPP. When ${K}_\gamma$ has moderate condition number after damping, this
approach can be cheaper than Cholesky.

\paragraph{Principle.}
Let ${K}_\gamma\succ 0$ and choose a scalar $\alpha>0$ such that the scaled matrix
\[
\widetilde{{K}}  \coloneqq  \alpha {K}_\gamma
\]
has spectrum contained in $(\varepsilon, 1]$, i.e.,
\[
\varepsilon  <  \lambda_{\min}(\widetilde{{K}})  \le  \lambda_{\max}(\widetilde{{K}})  \le  1
\]
for a small safety margin $\varepsilon>0$. Define the residual matrix
\[
{R}_0  \coloneqq  {I}_n-\widetilde{{K}}.
\]
With this scaling, the spectrum of ${R}_0$ lies in $[0, 1-\varepsilon]$. Since
$\widetilde{{K}}={I}_n-{R}_0$,
\[
{K}_\gamma^{-1} = \alpha \widetilde{{K}}^{-1}
 = \alpha ({I}_n-{R}_0)^{-1}.
\]
The goal is therefore to build an explicit approximation ${X}\approx({I}_n-{R}_0)^{-1}$.
We initialize the inverse estimate as
\[
{X}_0 = {I}_n.
\]
At step $k=1,\dots,T$, we pick a polynomial $q_k(\cdot)$ and apply it to the current residual matrix
\[
{Q}_k  \coloneqq  q_k({R}_{k-1}),
\qquad
{X}_k \coloneqq {X}_{k-1} {Q}_k.
\]
This induces a residual update. Using ${R}_{k-1}={I}_n-\widetilde{{K}}{X}_{k-1}$, we have
$\widetilde{{K}}{X}_{k-1}={I}_n-{R}_{k-1}$, hence
\[
{R}_k
 = 
{I}_n-\widetilde{{K}}{X}_k
 = 
{I}_n-\widetilde{{K}}{X}_{k-1} {Q}_k
 = 
{I}_n-({I}_n-{R}_{k-1}) q_k({R}_{k-1})
 = 
\phi_k({R}_{k-1}),
\]
where the induced scalar map on eigenvalues is
\[
r^{+} = \phi_k(r) \coloneqq 1-(1-r) q_k(r).
\]
Crucially, the admissible interval for the residual is iteration dependent. If ${R}_{k-1}$ has spectrum contained
in an interval $r\in\mathcal{I}_{k-1}$ ($\mathcal{I}_0=[0, 1-\varepsilon]$), then the next spectrum is contained in
\[
r^{+}\in \phi_k(\mathcal{I}_{k-1}),
\]
and the interval shrinks as $k$ grows. As the residual contracts toward zero, the effective design interval becomes much
smaller than $[0, 1-\varepsilon]$, which allows later polynomials to be optimized for a tighter range and achieve stronger
contraction per SYPP.

\paragraph{Convergence guarantee.}
Define $s_0\coloneqq 1-\varepsilon$, so that the initial residual interval is $\mathcal{I}_0=[0,s_0]$. If the first polynomial $q_1$ is chosen so that
\[
\sup_{r\in \mathcal{I}_0} |\phi_1(r)|  \le  s_1  <  s_0,
\]
and for each later step $k\ge 2$ the polynomial $q_k$ is chosen so that
\[
\sup_{|r|\le s_{k-1}} |\phi_k(r)|  \le  s_k  <  s_{k-1},
\]
then $\|{R}_k\|_2\le s_k$ and the residual interval shrinks monotonically to $0$. Repeating a short
sequence of such contractive steps yields ${R}_T$ close to zero and ${X}_T$ as an explicit
approximation to $({I}_n-{R}_0)^{-1}=\widetilde{{K}}^{-1}$. Finally,
\[
{K}_\gamma^{-1} \approx \alpha {X}_T.
\]

\paragraph{How the polynomials are chosen.}
At the first step we choose $q_1$ to minimize $s_1=\sup_{r\in\mathcal{I}_0}|\phi_1(r)|$ on the one-sided interval $\mathcal{I}_0=[0,s_0]$. For each later step $k\ge 2$, we take the current symmetric spectral bound $s_{k-1}$ and choose $q_k$ to minimize
the next bound $s_k=\sup_{|r|\le s_{k-1}}|\phi_k(r)|$ under coefficient-magnitude constraints and an explicit robustness margin
that accounts for finite precision and modeling error. We implement this minimax design on a grid (via a linear
program) and then chain multiple steps under a total SYPP budget (via dynamic programming). The first step is designed on a
one-sided interval, and subsequent steps are designed on the symmetric interval to remain stable once roundoff introduces
small negative eigenvalues.

\paragraph{Numerical safeguards for precision.}
Two hyperparameters in the LP directly control how conservative the bound update $s_{k-1}\mapsto s_k$ is.
First, \texttt{INTERVAL\_PAD\_REL} enlarges the design interval. For the first step, instead of optimizing $\phi_1$ only on $\mathcal{I}_0=[0,s_0]$, the solver designs on
\[
\mathcal{I}_0^{\mathrm{design}}=[0,(1+\texttt{INTERVAL\_PAD\_REL})s_0].
\]
For each later step $k\ge 2$, instead of optimizing $\phi_k$ only on the current spectral
bound $|r|\le s_{k-1}$, the solver designs on $|r|\le s_k^{\mathrm{design}}$ with
\[
s_k^{\mathrm{design}}  =  (1+\texttt{INTERVAL\_PAD\_REL}) s_{k-1}.
\]
This guards against underestimating the current residual bound by enforcing contraction on a slightly wider interval than
the nominal bound.
Second, \texttt{NOISE\_ABS} is a worst-case absolute perturbation radius for the scalar residual map. Concretely, the LP is
solved under the robust requirement that for all $r$ in the design interval and for all perturbations
$\delta\in[-\texttt{NOISE\_ABS}, \texttt{NOISE\_ABS}]$, the perturbed update remains bounded
\[
\bigl|\phi_k(r)+\delta\bigr|  \le  \gamma_k.
\]
Equivalently, the solver enforces $|\phi_k(r)|\le \gamma_k-\texttt{NOISE\_ABS}$ on the grid, and then reports the robust bound
as $s_k=\gamma_k=\sup|\phi_k(r)|+\texttt{NOISE\_ABS}$. 
The algorithm is shown in Algorithm~\ref{alg:poly-right-precond}.

\begin{algorithm}[htb]
\caption{Inverse of ${K}_\gamma={K}+\gamma{I}_n$ via polynomial approximation}
\label{alg:poly-right-precond}
\DontPrintSemicolon

\KwIn{positive semidefinite matrix ${K}\in\mathbb{R}^{n\times n}$, damping $\gamma>0$, SYPP budget $B$.}
\KwOut{Explicit approximation $\widetilde{{K}_\gamma^{-1}}\approx({K}+\gamma{I}_n)^{-1}$.}

${K}_\gamma \gets {K}+\gamma{I}_n$\;

Estimate any upper bound of the largest eigenvalue $\overline{\lambda}_{\max}({K})\ge \lambda_{\max}({K})$\;

$\overline{\lambda}_{\max}({K}_\gamma)\gets \overline{\lambda}_{\max}({K})+\gamma$\;

$\alpha \gets 1/\overline{\lambda}_{\max}({K}_\gamma)$,\quad
$\widetilde{{K}} \gets \alpha {K}_\gamma$,\quad
$\bar{\varepsilon} \gets \alpha \gamma$,\quad
${X}_0 \gets {I}_n$,\quad
${R}_0 \gets {I}_n-\widetilde{{K}}$\;

Choose a plan $(q_1,\dots,q_T)$ from Table~\ref{tab:poly-hp-plans} with tabulated $\varepsilon\le \bar{\varepsilon}$ and total SYPP cost $\le B$\;

\For{$k\gets 1$ \KwTo $T-1$}{
  ${Q}_k \gets q_k({R}_{k-1})$\;
  ${X}_k \gets {X}_{k-1} {Q}_k$\;
  ${R}_k \gets {I}_n-({I}_n-{R}_{k-1}) {Q}_k$\;
}

${Q}_T \gets q_T({R}_{T-1})$\;
${X}_T \gets {X}_{T-1} {Q}_T$\;

\Return $\widetilde{{K}_\gamma^{-1}} \gets \alpha {X}_T$ such that $\left\|{I}_n-{K}_\gamma \widetilde{{K}_\gamma^{-1}}\right\|_2\le s_{\mathrm{out}}$, where $s_{\mathrm{out}}$ is the certified residual bound associated with the selected plan\;
\end{algorithm}

\begin{table}[htbp]
  \centering
  \small
  \begin{tabular}{r r r r l}
    \hline
    $\varepsilon$ & Total & $s_{\mathrm{out}}$ & SYPP & Polynomial $q_k(x)$ \\
    \hline
    0.0015 & 12 & 0.030717 & 3 & $q_1(x)=1.991037-15.856588x+31.760959x^2$ \\
           &    &          & 4 & $q_2(x)=0.102569+0.102569x+7.383161x^2+7.383161x^3$ \\
           &    &          & 3 & $q_3(x)=1+2.541910x+2.541910x^2$ \\
           &    &          & 2 & $q_4(x)=1+1.192261x+1.192261x^2$ \\
    \hline
    0.0015 & 13 & 0.004865 & 3 & $q_1(x)=1.991037-15.856588x+31.760959x^2$ \\
           &    &          & 3 & $q_2(x)=1+3.839962x+3.839963x^2$ \\
           &    &          & 3 & $q_3(x)=1+2.989700x+2.989700x^2$ \\
           &    &          & 2 & $q_4(x)=1.244063+1.244063x$ \\
           &    &          & 2 & $q_5(x)=1+1.047265x+1.047265x^2$ \\
    \hline
    0.003 & 10 & 0.019885 & 3 & $q_1(x)=1.964953-15.439061x+31.064790x^2$ \\
          &    &          & 3 & $q_2(x)=1+3.346712x+3.346712x^2$ \\
          &    &          & 2 & $q_3(x)=1.403255+1.403255x$ \\
          &    &          & 2 & $q_4(x)=1+1.140006x+1.140006x^2$ \\
    \hline
    0.003 & 11 & 0.002839 & 3 & $q_1(x)=1.964953-15.439061x+31.064790x^2$ \\
          &    &          & 3 & $q_2(x)=1+3.346712x+3.346712x^2$ \\
          &    &          & 3 & $q_3(x)=1+1.757644x+1.757644x^2$ \\
          &    &          & 2 & $q_4(x)=1+1.028634x+1.028634x^2$ \\
    \hline
    0.006 &  8 & 0.047094 & 3 & $q_1(x)=1.915935-14.653845x+29.754543x^2$ \\
          &    &          & 3 & $q_2(x)=1+2.716205x+2.716205x^2$ \\
          &    &          & 2 & $q_3(x)=1+1.262596x+1.262596x^2$ \\
    \hline
    0.006 &  9 & 0.014106 & 3 & $q_1(x)=1.915935-14.653845x+29.754543x^2$ \\
          &    &          & 4 & $q_2(x)=0.639753+0.639753x+4.060692x^2+4.060692x^3$ \\
          &    &          & 2 & $q_3(x)=1+1.108734x+1.108734x^2$ \\
    \hline
    0.006 & 10 & 0.002087 & 3 & $q_1(x)=1.915935-14.653845x+29.754543x^2$ \\
          &    &          & 3 & $q_2(x)=1+2.716205x+2.716205x^2$ \\
          &    &          & 2 & $q_3(x)=1.160973+1.160973x$ \\
          &    &          & 2 & $q_4(x)=1+1.020113x+1.020111x^2$ \\
    \hline
    0.012 &  7 & 0.047594 & 3 & $q_1(x)=1.828900-13.257429x+27.420730x^2$ \\
          &    &          & 3 & $q_2(x)=1+2.072900x+2.072900x^2$ \\
          &    &          & 1 & $q_3(x)=1.046594+1.046594x$ \\
    \hline
    0.012 &  8 & 0.008118 & 3 & $q_1(x)=1.828900-13.257429x+27.420730x^2$ \\
          &    &          & 3 & $q_2(x)=1+2.072900x+2.072900x^2$ \\
          &    &          & 2 & $q_3(x)=1+1.071558x+1.071558x^2$ \\
    \hline
    0.025 &  6 & 0.048057 & 4 & $q_1(x)=1.528164+1.400800x-12.902311x^2+32x^4$ \\
          &    &          & 2 & $q_2(x)=1+1.266514x+1.266514x^2$ \\
    \hline
    0.025 &  7 & 0.008458 & 3 & $q_1(x)=1.679044-10.844625x+23.373926x^2$ \\
          &    &          & 2 & $q_2(x)=1.301562+1.301562x$ \\
          &    &          & 2 & $q_3(x)=1+1.073878x+1.073878x^2$ \\
    \hline
  \end{tabular}
  \caption{Polynomial plans with \texttt{CMAX}=32 (the maximum absolute polynomial coefficient), \texttt{INTERVAL\_PAD\_REL}=0.001, and \texttt{NOISE\_ABS}=0.001. Total denotes the total number of symmetric polynomial products (SYPP) used by the plan, and the SYPP column reports the SYPP cost of each polynomial $q_k$. The quantity $s_{\mathrm{out}}$ is the certified final residual bound, i.e., an upper bound on $\left\|{I}_n-{K}_\gamma \widetilde{{K}_\gamma^{-1}}\right\|_2$.}
  \label{tab:poly-hp-plans}
\end{table}

\section{CIFAR-10 Experiment}
\label{app:cifar10-details}

\paragraph{Dataset and model.}
This appendix describes the CIFAR-10~\citep{krizhevsky2009learning} experiment shown in the second row of Figure~\ref{fig:record4-baseline-plot}, where we report test accuracy versus training step and training time. The 50{,}000 training images are split into 45{,}000 training examples and 5{,}000 validation examples, and the standard 10{,}000-image test set is used for final evaluation. Training augmentation consists of random cropping with padding 4 and random horizontal flipping. All images are normalized channelwise using the standard CIFAR-10 mean and standard deviation. The model is a residual MLP with 32 hidden layers of width 512. Each hidden layer consists of a linear map followed by LayerNorm and GELU. Residual connections are used whenever the input and output dimensions match. Thus, the first layer $3072 \to 512$ has no skip connection, while the remaining 31 hidden layers of shape $512 \to 512$ use residual additions. The output layer is a linear classifier from 512 to 10.

\paragraph{Training setup.}
All three methods are trained for 100 epochs with batch size 4096 on an A100 GPU. The learning-rate schedule is linear warmup followed by cosine decay, with 100 warmup steps and minimum learning-rate ratio 0.1. Validation is evaluated every 24 training steps during hyperparameter tuning. For Muon and Newton--Muon, only the hidden-layer weight matrices are assigned to the matrix optimizer, while the remaining parameters are optimized by AdamW. For the pure AdamW baseline, all trainable parameters are optimized by AdamW. We tuned the hyperparameters of all methods on the validation split. We then retrained each method with its selected hyperparameters on the full 50{,}000-image CIFAR-10 training set and report test accuracy versus training step and training time in Figure~\ref{fig:record4-baseline-plot}.

\paragraph{Final hyperparameters.}
The final AdamW baseline uses learning rate $8\times 10^{-4}$, weight decay $10^{-2}$, and $(\beta_1,\beta_2)=(0.9,0.999)$. The final Muon configuration uses AdamW on the non-matrix parameters with learning rate $1.6\times 10^{-3}$, weight decay $10^{-2}$, and $(\beta_1,\beta_2)=(0.9,0.999)$, and applies Muon to the hidden-layer weight matrices with matrix learning rate $0.16$, matrix weight decay $10^{-3}$, momentum $0.8$. The final Newton--Muon configuration uses AdamW on the non-matrix parameters with learning rate $8\times 10^{-4}$, weight decay $10^{-2}$, and $(\beta_1,\beta_2)=(0.9,0.999)$, and applies Newton--Muon to the hidden-layer weight matrices with matrix learning rate $0.16$, matrix weight decay $3\times 10^{-4}$, momentum $0.75$, together with EWMA $\beta=0.95$, ridge $\gamma=0.05$, and refresh interval $k=16$.

\section{Kronecker-Factored Curvature}
\label{app:kronecker-factor}

Here we derive the Kronecker-factored approximation in~\eqref{eq:HW-kron}.
In a transformer with sequence length $N$, let $\boldsymbol{z}_t\in\R^{n}$ denote the input activation at token $t$ and let $\boldsymbol{y}_t\in\R^{m}$ denote the output of a linear map $\boldsymbol{y}_t={W}\boldsymbol{z}_t$ for $t=1,\dots,N$. For an update direction ${Q}\in\R^{m\times n}$, the output perturbation at token $t$ is $\Delta \boldsymbol{y}_t={Q}\boldsymbol{z}_t$.
To expose the Kronecker structure, first consider a single token loss $L_t(\boldsymbol{y}_t)$ with output-space curvature ${H}_t \coloneqq \nabla_{\boldsymbol{y}_t}^2 L_t(\boldsymbol{y}_t)\in\R^{m\times m}$.
Then the exact second-order change satisfies
\[
  \delta^2 L_t
  = (\Delta \boldsymbol{y}_t)^\top {H}_t (\Delta \boldsymbol{y}_t)
  = ({Q}\boldsymbol{z}_t)^\top {H}_t ({Q}\boldsymbol{z}_t).
\]
Using $\vecop({Q}\boldsymbol{z}_t)=(\boldsymbol{z}_t^\top\otimes {I}_m)\vecop({Q})$, we obtain
\[
  ({Q}\boldsymbol{z}_t)^\top {H}_t ({Q}\boldsymbol{z}_t)
  =
  \vecop({Q})^\top\Big((\boldsymbol{z}_t\boldsymbol{z}_t^\top)\otimes {H}_t\Big)\vecop({Q}).
\]
For token-coupled losses, stack token outputs into $\boldsymbol{y} \coloneqq [\boldsymbol{y}_1^\top,\dots,\boldsymbol{y}_N^\top]^\top \in \R^{mN}$,
and let the averaged scalar loss be $f(\boldsymbol{y}) \coloneqq L(\boldsymbol{y})/N$, where $L(\boldsymbol{y})$ denotes the summed loss over the $N$ tokens. The token-coupled output-space curvature of $L$ is
\[
  \mathscr{H}  \coloneqq  \nabla_{\boldsymbol{y}}^2 L(\boldsymbol{y})\in\R^{(mN)\times(mN)}.
\]
Write $\mathscr{H}$ in $N\times N$ blocks of size $m\times m$:
\[
  \mathscr{H} =
  \begin{bmatrix}
    {H}_{11} & \cdots & {H}_{1N}\\
    \vdots & \ddots & \vdots\\
    {H}_{N1} & \cdots & {H}_{NN}
  \end{bmatrix},
  \qquad
  {H}_{ts}\in\R^{m\times m}.
\]
Then
\[
  \delta^2 f
   = 
  \frac{1}{N}\sum_{t=1}^N\sum_{s=1}^N ({Q}\boldsymbol{z}_t)^\top {H}_{ts} ({Q}\boldsymbol{z}_s).
\]
Each term can be written as
\[
  ({Q}\boldsymbol{z}_t)^\top {H}_{ts} ({Q}\boldsymbol{z}_s)
  =
  \vecop({Q})^\top\Big((\boldsymbol{z}_t \boldsymbol{z}_s^\top)\otimes {H}_{ts}\Big)\vecop({Q}),
\]
so the exact parameter-space Hessian for the token-coupled loss is
\begin{equation}
  \label{eq:HW-transformer-exact}
  {\mathcal{H}}_{{W}}
   = 
  \frac{1}{N}\sum_{t=1}^N\sum_{s=1}^N (\boldsymbol{z}_t \boldsymbol{z}_s^\top)\otimes {H}_{ts}
  \qquad\in\R^{(mn)\times(mn)}.
\end{equation}
Define the input activation matrix ${Z}=[\boldsymbol{z}_1,\dots,\boldsymbol{z}_N]\in\R^{n\times N}$,
the diagonal-block average curvature
\[
  {H}  \coloneqq  \frac{1}{N}\sum_{t=1}^N {H}_{tt}.
\]
A simplifying approximation retains only diagonal-token curvature contributions and replaces ${H}_{tt}$ by their average ${H}$,
yielding
\begin{equation}
  \label{eq:HW-transformer-kron}
  {\mathcal{H}}_{{W}}
   \approx 
  \frac{1}{N}\sum_{t=1}^N (\boldsymbol{z}_t \boldsymbol{z}_t^\top)\otimes {H}
   = 
  ({Z}{Z}^\top/N)\otimes {H}.
\end{equation}

\paragraph{Remark.} 
For architectures where loss decomposes across tokens/examples such as MLPs, the parameter-space Hessian~\eqref{eq:HW-transformer-exact} is block-diagonal, i.e., ${H}_{ts}=\mathbf{0}$ for $t\neq s$.
In this case, \eqref{eq:HW-transformer-exact} reduces to the diagonal-token form
\[
  {\mathcal{H}}_{{W}}
  =
  \frac{1}{N}\sum_{t=1}^N (\boldsymbol{z}_t \boldsymbol{z}_t^\top)\otimes {H}_{tt},
\]
so the only approximation in \eqref{eq:HW-transformer-kron} comes from replacing the varying blocks ${H}_{tt}$ by their average ${H}$.
As a result, the Kronecker estimator \eqref{eq:HW-transformer-kron} is expected to be more accurate when the loss is not token-coupled.

\section{Theoretical Quadratic Score Under Isotropic Activation}
\label{app:quadratic-score}
This appendix develops the theoretical quadratic score analysis under an isotropic activation second moment, deriving explicit formulas and approximations for gradient descent, Muon, Newton--Muon, and Newton, and then using them to interpret the corresponding numerical study.

\subsection{Score Formulas for Different Directions}
\label{app:score-formulas}

We use \eqref{eq:score-matrix-form} to compare four update directions under the quadratic surrogate:
(i) the raw GD direction ${Q}={G}$, (ii) the Muon direction ${Q}_{\mathrm{Muon}}=\mathrm{msgn}({G})$, (iii) the Newton--Muon direction ${Q}_{\mathrm{Newton\text{-}Muon}}=\mathrm{msgn}\big({G}({Z}{Z}^\top)^{-1}\big)$, and (iv) the Newton direction ${Q}_{\mathrm{Newton}}={H}^{-1}{G}({Z}{Z}^\top)^{-1}$.

\paragraph{GD direction.}
For ${Q}={G}$, the numerator is $\mathrm{tr}({G}{G}^\top)=\|{G}\|_F^2$ and the denominator is
$\mathrm{tr} \big({H} {G} ({Z}{Z}^\top/N) {G}^\top\big)$. Using ${G}={H} ({W}-W^\star) ({Z}{Z}^\top/N)$,
we have
\[
  \mathrm{tr}({G}{G}^\top)
   = 
  \mathrm{tr} \Big({H}^2 ({W}-W^\star) ({Z}{Z}^\top/N)^2 ({W}-W^\star)^\top\Big).
\]
Similarly,
\[
  \mathrm{tr} \big({H} {G} ({Z}{Z}^\top/N) {G}^\top\big)
   = 
  \mathrm{tr} \Big({H}^3 ({W}-W^\star) ({Z}{Z}^\top/N)^3 ({W}-W^\star)^\top\Big).
\]
Hence the score for the GD direction has the closed form
\begin{equation}
  \label{eq:score-grad-deterministic}
  s ({G})
   = 
  \frac{
    \mathrm{tr} \Big({H}^2 ({W}-W^\star) ({Z}{Z}^\top/N)^2 ({W}-W^\star)^\top\Big)^2
  }{
    \mathrm{tr} \Big({H}^3 ({W}-W^\star) ({Z}{Z}^\top/N)^3 ({W}-W^\star)^\top\Big)
  }.
\end{equation}

\paragraph{Muon direction.}
We consider the matrix sign \( {Q}_{\mathrm{Muon}}  \coloneqq  \mathrm{msgn}(G) \). For $G=USV^\top$,
\[
  \mathrm{tr}({Q}_{\mathrm{Muon}}{G}^\top)
   = 
  \mathrm{tr} \big({U}{V}^\top ({V}S {U}^\top)\big)
   = 
  \mathrm{tr}(S)
   = 
  \|{G}\|_\ast
   = 
  \big\|{H} ({W}-W^\star) ({Z}{Z}^\top/N)\big\|_\ast,
\]
and the denominator becomes
\[
  \mathrm{tr} \Big({H} {Q}_{\mathrm{Muon}} ({Z}{Z}^\top/N) {Q}_{\mathrm{Muon}}^\top\Big)
   = 
  \mathrm{tr} \Big({H} {U} {V}^\top({Z}{Z}^\top/N){V} {U}^\top\Big)
   = 
  \mathrm{tr} \Big({U}^\top {H} {U}  \cdot  {V}^\top({Z}{Z}^\top/N){V}\Big).
\]
Therefore, the Muon-direction score is
\begin{equation}
  \label{eq:score-muon-deterministic}
  s ({Q}_{\mathrm{Muon}})
   = 
  \frac{
    \big\|{H} ({W}-W^\star) ({Z}{Z}^\top/N)\big\|_\ast^{ 2}
  }{
    \mathrm{tr} \Big({U}^\top {H} {U}  \cdot  {V}^\top({Z}{Z}^\top/N){V}\Big)
  }.
\end{equation}
The score for Newton--Muon can be derived similarly. We do not analyze it theoretically here; instead, we evaluate its predicted score numerically and compare it with Muon in our numerical study.

\paragraph{Newton direction.}
The Newton step is \( {Q}_{\mathrm{Newton}} \propto {H}^{-1} {G} ({Z}{Z}^\top)^{-1} \). Substituting into
\eqref{eq:score-matrix-form} gives
\[
  \mathrm{tr}({Q}_{\mathrm{Newton}}{G}^\top)
   = 
  \mathrm{tr} \big({H}^{-1}{G}({Z}{Z}^\top)^{-1}{G}^\top\big),
\]
and, using ${H} {Q}_{\mathrm{Newton}}({Z}{Z}^\top/N)={G}/N$,
\[
  \mathrm{tr} \Big({H} {Q}_{\mathrm{Newton}} ({Z}{Z}^\top/N) {Q}_{\mathrm{Newton}}^\top\Big)
   = 
  \frac{1}{N}\mathrm{tr} \big({G} {Q}_{\mathrm{Newton}}^\top\big)
   = 
  \frac{1}{N}\mathrm{tr}({Q}_{\mathrm{Newton}}{G}^\top).
\]
Therefore, \( s ({Q}_{\mathrm{Newton}})
   = 
  N \mathrm{tr}({Q}_{\mathrm{Newton}}{G}^\top)\).
Under ${G}={H} ({W}-W^\star) ({Z}{Z}^\top/N)$, this can be rewritten as
\begin{equation}
  \label{eq:score-newton-displacement}
  s ({Q}_{\mathrm{Newton}})
   = 
  \mathrm{tr} \Big({H} ({W}-W^\star) ({Z}{Z}^\top/N) ({W}-W^\star)^\top\Big).
\end{equation}

\subsection{Isotropic Baseline Numerical Study}

From \eqref{eq:score-grad-deterministic}, \eqref{eq:score-muon-deterministic}, and \eqref{eq:score-newton-displacement}, the quadratic score
\eqref{eq:score-matrix-form} yields explicit expressions for the GD, Muon, and Newton directions. By using different distributions on ${H}$, ${Z}$, and $({W}-W^\star)$, we can evaluate how
anisotropy in curvature, activations, and displacement interacts to favor different update geometries, either analytically when expectations simplify or
numerically by Monte Carlo simulation.
Here we study the simplest baseline in which both the activation and the displacement are isotropic Gaussian; note that the same framework can be reused under more structured, non-isotropic
specifications by changing only the distribution for ${H}$, ${Z}$, and $({W}-W^\star)$.

\paragraph{Assumptions.}
First, we sample ${Z}$ with i.i.d.\ standard normal entries. For the theoretical analysis, we assume the activation second moment is isotropic (while in the simulations we use a finite sample size):
\begin{equation}
  \label{eq:assump-AAId}
  {Z}{Z}^\top/N = {I}_n.
\end{equation}
Second, we assume the square case \(m=n\). Third, we further assume the displacement matrix \({W}-W^\star\) has i.i.d.\ standard normal entries. Under \eqref{eq:assump-AAId}, \({G}={H}({W}-W^\star)\) and the score \eqref{eq:score-matrix-form} simplifies to \(s({Q})=\mathrm{tr}({Q}{G}^\top)^2/\mathrm{tr}\big({H}{Q}{Q}^\top\big)\). Let \({D}_{{H}}\coloneqq\mathrm{diag}(\lambda_1,\dots,\lambda_m)\). We fix this diagonal spectrum (to analyze different levels of curvature anisotropy through the choice of \(\{\lambda_k\}\)), then sample a random orthogonal matrix \({P}\in\R^{m\times m}\) and define \({H}\coloneqq {P}{D}_{{H}}{P}^\top\).

\paragraph{GD score.}
Under \eqref{eq:assump-AAId}, \eqref{eq:score-grad-deterministic} reduces to
\begin{equation}
  \label{eq:score-grad-AAId}
  s ({G})
   = 
  \frac{
    \mathrm{tr} \Big({H}^2 ({W}-W^\star)({W}-W^\star)^\top\Big)^2
  }{
    \mathrm{tr} \Big({H}^3 ({W}-W^\star)({W}-W^\star)^\top\Big)
  }.
\end{equation}
Moreover, for $k\in\{2,3\}$,
\[
  \mathbb{E} \left[\mathrm{tr} \Big({H}^k ({W}-W^\star)({W}-W^\star)^\top\Big)\right]
   = 
  \mathrm{tr} \Big({H}^k \mathbb{E}\big[({W}-W^\star)({W}-W^\star)^\top\big]\Big)
   = 
  n \mathrm{tr}({H}^k).
\]
In the large-$n$ regime, the random traces above concentrate around their means, so a standard
spectrum-level approximation to \eqref{eq:score-grad-AAId} is
\begin{equation}
  \label{eq:score-grad-spectrum-approx}
  s ({G})
   \approx 
  n \frac{\mathrm{tr}({H}^2)^2}{\mathrm{tr}({H}^3)}
   = 
  n \frac{\left(\sum_{j=1}^{m}\lambda_j^2\right)^2}{\sum_{j=1}^{m}\lambda_j^3}.
\end{equation}

\paragraph{Muon score.}
In the square full-rank case, \({U}\in\R^{m\times m}\) is orthogonal, so \(\mathrm{tr}({U}^\top {H} {U})=\mathrm{tr}({H})\). Thus
\begin{equation}
  \label{eq:score-muon-AAId-simple}
  s({Q}_{\mathrm{Muon}})
  =
  \frac{\|{H}({W}-W^\star)\|_\ast^{2}}{\mathrm{tr}({U}^\top {H} {U})}
  =
  \frac{\|{H}({W}-W^\star)\|_\ast^{2}}{\mathrm{tr}({H})}.
\end{equation}
While $\mathbb{E}\|{H} ({W}-W^\star)\|_\ast^{ 2}$ does not simplify to a closed form in general,
it admits an accurate high-dimensional deterministic approximation that depends only on the spectrum of ${H}$
and can be computed numerically. Let
\[
  {S}
   \coloneqq 
  \frac{1}{m} {H} ({W}-W^\star)({W}-W^\star)^\top {H}
   \in\R^{m\times m}.
\]
The eigenvalues of ${S}$ are the squared singular values of ${H}({W}-W^\star)$ divided by $m$.
Consequently,
\begin{equation}
  \label{eq:nucnorm-from-S}
  \|{H} ({W}-W^\star)\|_\ast
   = 
  \sum_{i=1}^m \sigma_i \big({H}({W}-W^\star)\big)
   = 
  \sqrt{m}\sum_{i=1}^m \sqrt{\lambda_i({S})}.
\end{equation}
In the regime $m=n\to\infty$ with ${H}$ deterministic and ${W}-W^\star$ having i.i.d.\ standard normal entries,
Silverstein--Choi's analysis of sample-covariance type matrices implies that the empirical spectral distribution of ${S}$
converges to a nonrandom limiting law characterized by a fixed-point equation for its Stieltjes transform
\citep{silverstein1995empirical}. In particular, writing the eigenvalues of ${H}$ as $\{\lambda_k\}_{k=1}^{m}$, the Stieltjes transform $m_S(z)$ of the limiting law is the unique solution in the
upper half-plane to the equation
\begin{equation}
  \label{eq:silverstein-choi-fixedpoint}
  m_S(z)
   = 
  - \frac{1}{z}\cdot\frac{1}{m}\sum_{k=1}^{m}\frac{1}{1+\lambda_k^2 m_S(z)},
  \qquad z\in\mathbb{C}_+.
\end{equation}
Given $m_S(z)$, the limiting spectral density $\rho_S(x)$ is obtained by the standard inversion formula
\begin{equation}
  \label{eq:stieltjes-inversion}
  \rho_S(x)
   = 
  \lim_{\eta\downarrow 0}\frac{1}{\pi} \Im  m_S(x+i\eta),
  \qquad x>0.
\end{equation}
Equations \eqref{eq:silverstein-choi-fixedpoint} and \eqref{eq:stieltjes-inversion} suggest a concrete numerical pipeline:
\begin{enumerate}
\item Choose a small $\eta>0$ and a grid of $x$ values covering the support of the spectrum of ${S}$.
\item For each $z=x+i\eta$, solve the fixed-point equation \eqref{eq:silverstein-choi-fixedpoint}.
\item Approximate $\rho_S(x)\approx (1/\pi)\Im m_S(x+i\eta)$.
\item Compute the limiting mean singular value of ${H}({W}-W^\star)$ via
\[
  \mu_{1/2}
   \coloneqq 
  \int_0^\infty \sqrt{x} \rho_S(x) \mathrm{d}x,
\]
using numerical quadrature on the grid.
\item Using \eqref{eq:nucnorm-from-S}, the leading-order scaling is
\begin{equation}
  \label{eq:nucnorm-asymp}
  \|{H}({W}-W^\star)\|_\ast
   \approx 
  m^{3/2} \mu_{1/2},
  \qquad
  \text{hence}\qquad
  \mathbb{E}\|{H}({W}-W^\star)\|_\ast^{ 2}
   \approx 
  m^{3} \mu_{1/2}^{ 2},
\end{equation}
where the approximation captures the dominant $m^3$ growth and depends only on the spectrum of ${H}$ through \eqref{eq:silverstein-choi-fixedpoint}.
\end{enumerate}
Substituting \eqref{eq:nucnorm-asymp} into \eqref{eq:score-muon-AAId-simple} yields a corresponding practical approximation for the
expected Muon score,
\begin{equation}
  \label{eq:muon-score-det-eq}
  \mathbb{E} s ({Q}_{\mathrm{Muon}})
   \approx 
  \frac{\mathbb{E}\|{H}({W}-W^\star)\|_\ast^{ 2}}{\mathrm{tr}({H})}
   \approx 
  \frac{m^{3} \mu_{1/2}^{ 2}}{\mathrm{tr}({H})}.
\end{equation}
We will use \eqref{eq:silverstein-choi-fixedpoint}--\eqref{eq:muon-score-det-eq} as the basis for numerically evaluating the Muon
numerator (and thus the score) from the eigenvalues of ${H}$.

\paragraph{Newton--Muon score.}
Under \eqref{eq:assump-AAId}, Newton--Muon coincides exactly with Muon:
\[
  {Q}_{\mathrm{Newton\text{-}Muon}}
   =
  \mathrm{msgn}\big({G}({Z}{Z}^\top)^{-1}\big)
   =
  \mathrm{msgn}({G})
   =
  {Q}_{\mathrm{Muon}}.
\]
Consequently, we use the same approximation
\begin{equation}
  \label{eq:newton-muon-score-det-eq}
  \mathbb{E} s ({Q}_{\mathrm{Newton\text{-}Muon}})
   \approx
  \frac{m^{3} \mu_{1/2}^{ 2}}{\mathrm{tr}({H})}.
\end{equation}

\paragraph{Newton score.}
Applying \eqref{eq:assump-AAId} to \eqref{eq:score-newton-displacement} yields
\[
  s ({Q}_{\mathrm{Newton}})
   = 
  \mathrm{tr} \Big({H} ({W}-W^\star)({W}-W^\star)^\top\Big).
\]
Taking expectation gives the exact identity
\begin{equation}
  \label{eq:Escore-newton}
  \mathbb{E} \left[s ({Q}_{\mathrm{Newton}})\right]
   = 
  \mathrm{tr} \Big({H}\cdot n{I}_m\Big)
   = 
  n \mathrm{tr}({H})
   = 
  n\sum_{k=1}^{m}\lambda_k.
\end{equation}

\paragraph{Comparison.}
Under the simplifying assumptions \eqref{eq:assump-AAId} and $m=n$, we can compare the three directions at the level of the
theoretical quantities already derived for Newton, GD, Muon, and Newton--Muon. Newton has an exact expected score
\eqref{eq:Escore-newton}:
\[
\mathbb{E} \left[s ({Q}_{\mathrm{Newton}})\right]
 = 
n \mathrm{tr}({H})
 = 
n\sum_{k=1}^{m}\lambda_k.
\]
For the GD direction, the spectrum-level approximation \eqref{eq:score-grad-spectrum-approx} gives
\[
s ({G})
 \approx 
n \frac{\mathrm{tr}({H}^2)^2}{\mathrm{tr}({H}^3)}
 = 
n \frac{\left(\sum_{k=1}^{m}\lambda_k^2\right)^2}{\sum_{k=1}^{m}\lambda_k^3}.
\]
For nonnegative eigenvalues $\lambda_k\ge 0$, apply Cauchy inequality to the two sequences $\lambda_k^{3/2}$ and $\lambda_k^{1/2}$:
\[
\frac{\left(\sum_k\lambda_k^2\right)^2}{\sum_k\lambda_k^3}
 \le 
\sum_k\lambda_k
\quad\Longrightarrow\quad
s ({G})
 \lesssim 
n\sum_k\lambda_k
=
\mathbb{E} \left[s ({Q}_{\mathrm{Newton}})\right].
\]
Equality in Cauchy holds if and only if $\lambda_k$ is constant across $k$. Thus, GD matches Newton only when ${H}\propto {I}_m$.

Newton--Muon and Muon also simplify in this equal-eigenvalues case, and its spectrum-only approximation \eqref{eq:muon-score-det-eq} becomes
fully explicit. When ${H}=\lambda {I}_m$ and $m=n$, the limiting law of
${S}=(1/m){H}({W}-W^\star)({W}-W^\star)^\top{H}$
is $\lambda^2$ times the Marchenko--Pastur law at aspect ratio $1$, which yields the closed form
\[
\mu_{1/2}
=
\frac{8}{3\pi}\lambda.
\]
Substituting into \eqref{eq:muon-score-det-eq} gives
\[
\mathbb{E} s ({Q}_{\mathrm{Muon}})
 \approx 
\frac{m^{3}\mu_{1/2}^{ 2}}{\mathrm{tr}({H})}
=
\frac{m^{3}\lambda^{2}\left(\frac{8}{3\pi}\right)^{ 2}}{m\lambda}
=
\left(\frac{64}{9\pi^2}\right)m^{2}\lambda
=
\left(\frac{64}{9\pi^2}\right)\mathbb{E} \left[s ({Q}_{\mathrm{Newton}})\right],
\]
so the isotropic asymptotic approximation predicts that Muon achieves a constant fraction of the Newton score (about $0.72$), while GD coincides with Newton.

Beyond ${H}=\lambda {I}_m$, the Muon and Newton--Muon approximation \eqref{eq:muon-score-det-eq} and \eqref{eq:newton-muon-score-det-eq} still depend only on the
spectrum of ${H}$ through $\mu_{1/2}$ (computed from \eqref{eq:silverstein-choi-fixedpoint} and \eqref{eq:stieltjes-inversion}),
but there is no longer a simple closed-form comparison between the Muon/Newton--Muon score \eqref{eq:muon-score-det-eq} and the GD
approximation \eqref{eq:score-grad-spectrum-approx}, as we will show next in our numerical study.

Under the isotropic-input theory \eqref{eq:assump-AAId}, the Newton--Muon and Muon theoretical predictions coincide exactly; any empirical separation between them comes solely from finite-sample deviations of ${Z}{Z}^\top/N$ from ${I}_n$.
In all experiments below we fix $\lambda_{\min}=10^{-4}$ and vary only $(N,p)$. For each setting, we run $1024$ independent simulations and report the mean score together with the 2.5\%--97.5\% interval. The theory marker in the plot denotes the theoretical score in the limit of infinite data and infinite dimensions $m,n\to\infty$.

\paragraph{Baseline ($N=8192$, $p=0.3$).}
Figure~\ref{fig:sim-baseline} shows a clearly anisotropic spectrum of ${H}$.
In this ill-conditioned regime, both Muon and Newton--Muon yield substantial score improvements over the raw
GD direction, and Newton--Muon is slightly better than Muon.

\begin{figure}[htb]
  \centering
  \begin{minipage}{0.48\linewidth}
    \centering
    \includegraphics[width=\linewidth]{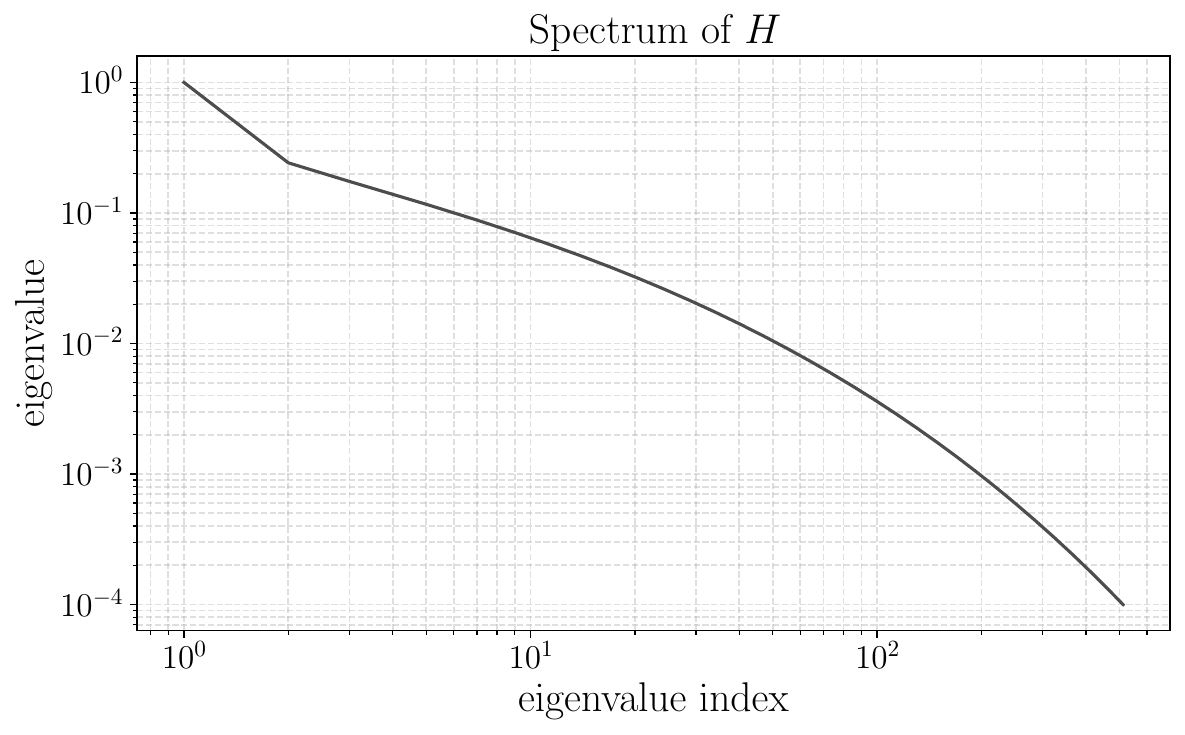}
  \end{minipage}\hfill
  \begin{minipage}{0.51\linewidth}
    \centering
    \includegraphics[width=\linewidth]{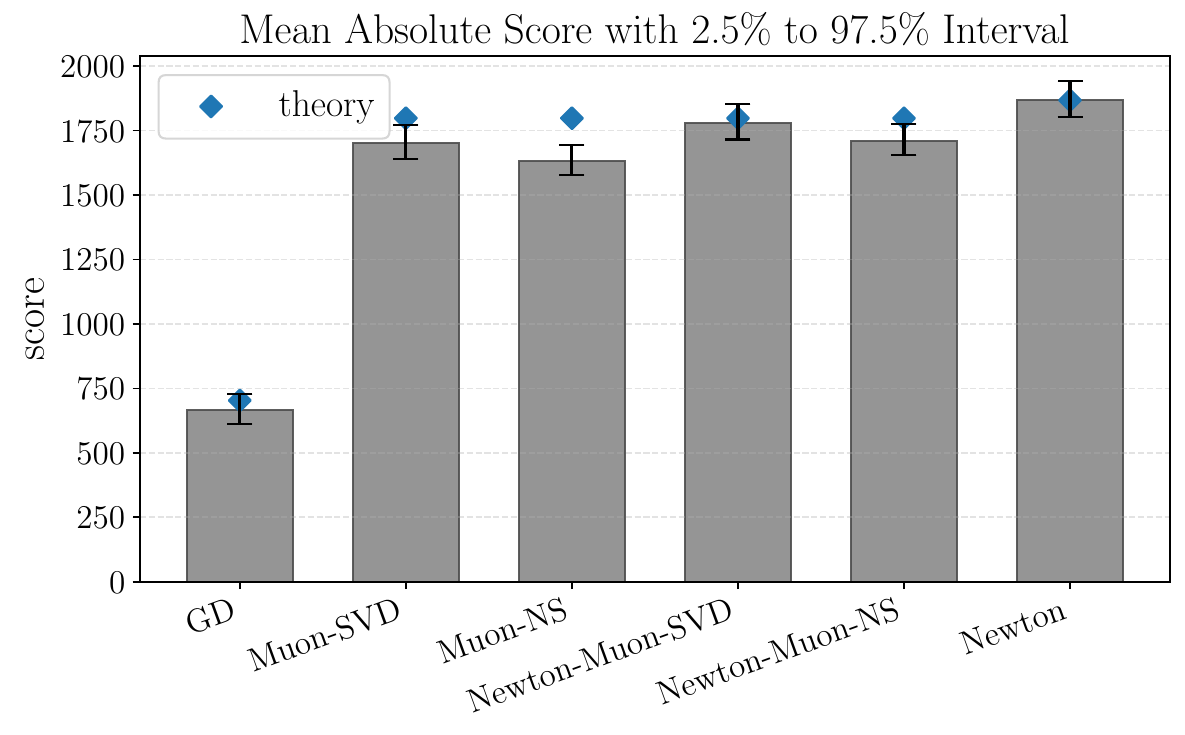}
  \end{minipage}
  \caption{Baseline configuration ($N=8192$, $p=0.3$): spectrum of ${H}$ (left) and mean absolute
  scores $s ({Q})$ for GD, Muon, Newton--Muon, and Newton (right).}
  \label{fig:sim-baseline}
\end{figure}

\paragraph{Uniform curvature ($N=8192$, $p=2.4$).}
Figure~\ref{fig:sim-uniformH} corresponds to a spectrum that is nearly flat at the top and drops sharply only near the tail.
In this more uniform-curvature setting, the score gaps among GD, Muon, and Newton--Muon narrow, indicating that Newton--Muon and Muon help when ${H}$ is strongly anisotropic. Here, Newton--Muon is still slightly better than Muon.

\begin{figure}[htb]
  \centering
  \begin{minipage}{0.48\linewidth}
    \centering
    \includegraphics[width=\linewidth]{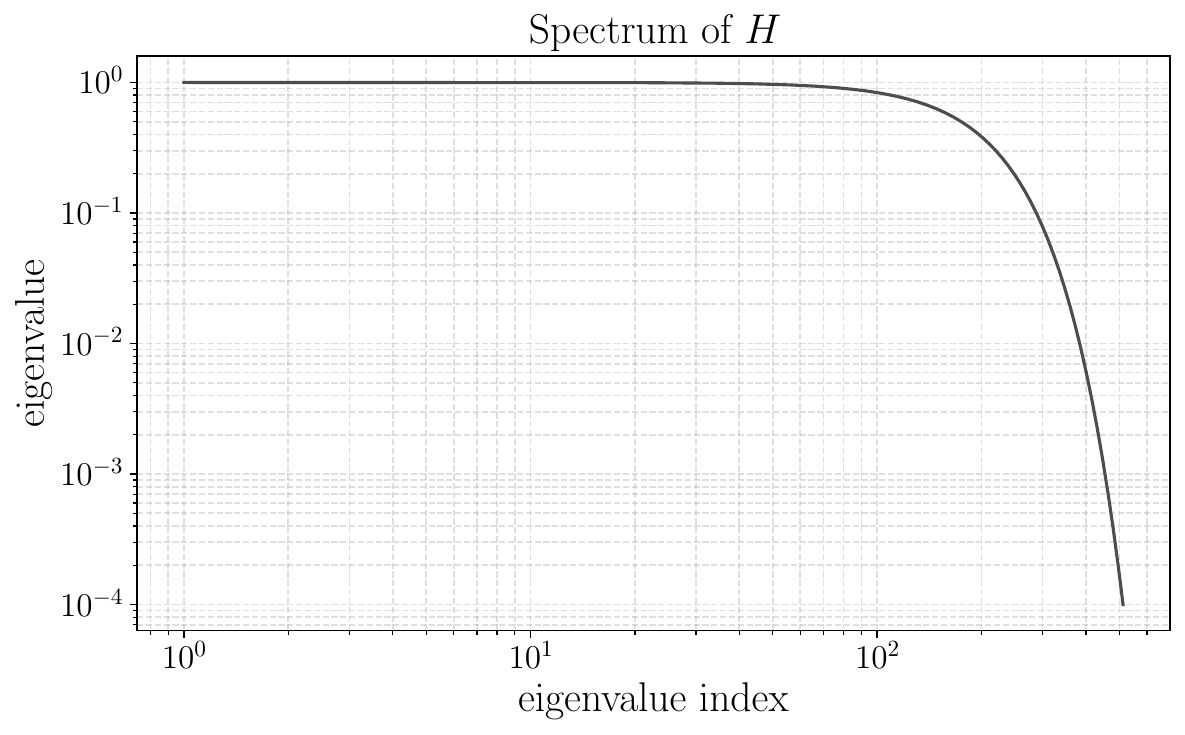}
  \end{minipage}\hfill
  \begin{minipage}{0.51\linewidth}
    \centering
    \includegraphics[width=\linewidth]{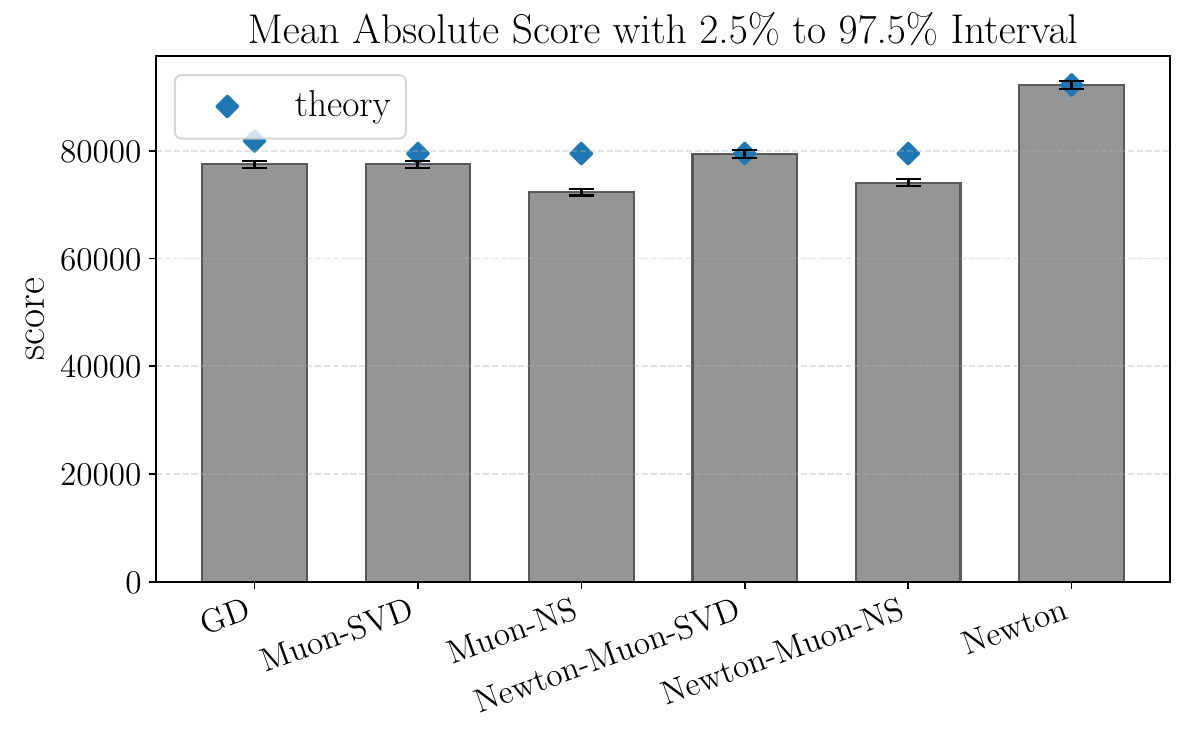}
  \end{minipage}
  \caption{More top-uniform curvature ($N=8192$, $p=2.4$): spectrum of ${H}$ (left) and mean absolute scores (right).}
  \label{fig:sim-uniformH}
\end{figure}

\paragraph{Smaller data ($N=1024$, $p=0.3$).}
Figure~\ref{fig:sim-smallN} keeps the same curvature shape as the baseline but uses a smaller sample size $N$ to form ${Z}{Z}^\top/N$, making the empirical second moment noisier. Relative to Figure~\ref{fig:sim-baseline}, the gap between Newton--Muon and Muon becomes much larger, suggesting that Newton--Muon more effectively compensates for activation anisotropy when the sample size is small.

\begin{figure}[htb]
  \centering
  \begin{minipage}{0.48\linewidth}
    \centering
    \includegraphics[width=\linewidth]{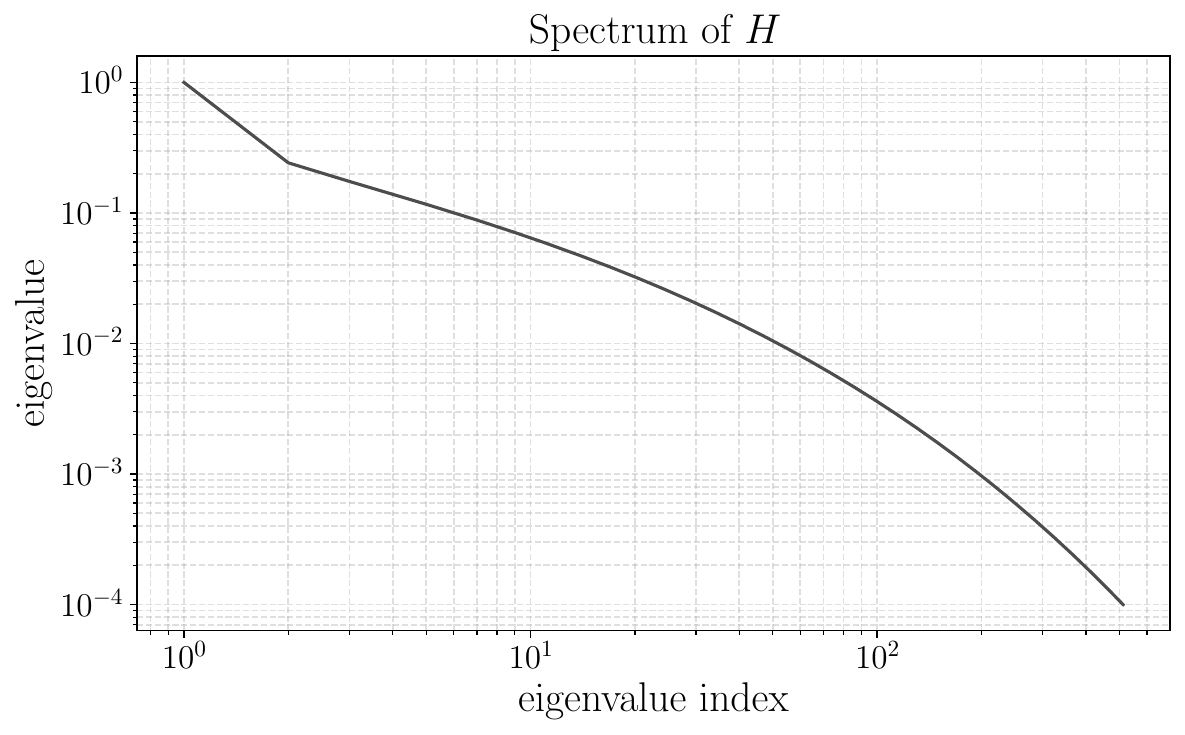}
  \end{minipage}\hfill
  \begin{minipage}{0.51\linewidth}
    \centering
    \includegraphics[width=\linewidth]{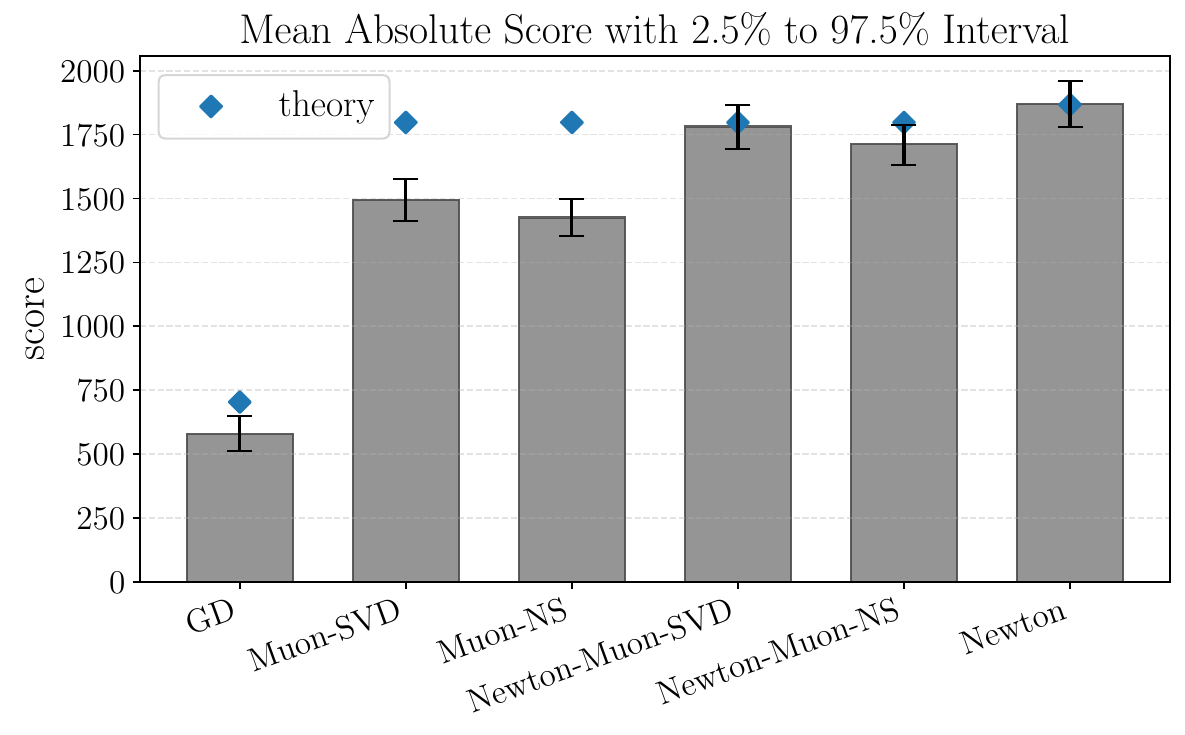}
  \end{minipage}
  \caption{Smaller-$N$ data ($N=1024$, $p=0.3$): spectrum of ${H}$ (left) and mean absolute scores (right).}
  \label{fig:sim-smallN}
\end{figure}

\paragraph{Conclusion.}
Across these three runs, Muon and Newton--Muon achieve their largest score gains when curvature is strongly anisotropic (Figure~\ref{fig:sim-baseline}); when ${H}$ becomes more uniform, the benefit shrinks or even vanishes (Figure~\ref{fig:sim-uniformH}). Newton--Muon outperforms Muon in these three experiments, and especially when $N$ is reduced so that ${Z}{Z}^\top/N$ is less isotropic, the advantage of Newton--Muon over Muon further increases (Figure~\ref{fig:sim-smallN}).

\section{Non-Isotropic Assumption}
\label{app:llm-exp-sigmaW}

Here we report attempts to replace the isotropic assumption $\Sigma_W \propto I_m$ by estimating $\Sigma_W$ from training dynamics. \eqref{eq:SigmaW-polar-update3} motivates the use of a non-isotropic ${\Sigma}_{W}$. When ${\Sigma}_{W}$ admits a factorization of the form ${\Sigma}_{W} = MM^\top$, the update in \eqref{eq:SigmaW-polar-update3} can be implemented without explicitly constructing $\Sigma_W^{1/2}$.

\begin{proposition}[Factorized form of \eqref{eq:SigmaW-polar-update3}]
\label{prop:factorized-polar-SigmaW}
Suppose ${\Sigma}_{{W}}={M}{M}^\top$ for some ${M}\in\mathbb{R}^{m\times r}$. Then
\[
{Q}^\star
=
{M} \mathrm{msgn}\Big({M}^\top {G}({Z}{Z}^\top)^{-1}\Big).
\]
\end{proposition}

\begin{proof}[Proof of Proposition~\ref{prop:factorized-polar-SigmaW}]
Let $k\coloneqq\mathrm{rank}(M)$ and take a compact SVD $M=U_M S_M V_M^\top$, where $U_M\in\mathbb{R}^{m\times k}$ and $V_M\in\mathbb{R}^{r\times k}$ have orthonormal columns and $S_M\in\mathbb{R}^{k\times k}$ is diagonal with positive entries. Then $\Sigma_W=MM^\top=U_M S_M^2 U_M^\top$, so $\Sigma_W^{1/2}=U_M S_M U_M^\top$.
Hence
\[
\Sigma_W^{1/2}G(ZZ^\top)^{-1}
=
U_M S_M U_M^\top G(ZZ^\top)^{-1}.
\]
Using the compact-SVD definition of the rectangular matrix sign, if $U_M$ has orthonormal columns, then $\mathrm{msgn}(U_M X)=U_M\mathrm{msgn}(X)$ for any matrix $X$. Applying this with $X=S_M U_M^\top G(ZZ^\top)^{-1}$ gives
\[
\mathrm{msgn}\Big(\Sigma_W^{1/2}G(ZZ^\top)^{-1}\Big)
=
U_M \mathrm{msgn}\Big(S_M U_M^\top G(ZZ^\top)^{-1}\Big).
\]
Substituting into \eqref{eq:SigmaW-polar-update3}, we obtain
\begin{equation}
\label{eq:factorized-polar-qstar}
Q^\star
=
\Sigma_W^{1/2}
\mathrm{msgn}\Big(\Sigma_W^{1/2}G(ZZ^\top)^{-1}\Big)
=
U_M S_M
\mathrm{msgn}\Big(S_M U_M^\top G(ZZ^\top)^{-1}\Big).
\end{equation}

On the other hand, $M^\top G(ZZ^\top)^{-1}=V_M S_M U_M^\top G(ZZ^\top)^{-1}$. Applying the same identity with $V_M$ yields
\[
\mathrm{msgn}\Big(M^\top G(ZZ^\top)^{-1}\Big)
=
V_M \mathrm{msgn}\Big(S_M U_M^\top G(ZZ^\top)^{-1}\Big).
\]
Therefore
\begin{align*}
M \mathrm{msgn}\Big(M^\top G(ZZ^\top)^{-1}\Big)
&= 
(U_M S_M V_M^\top)\Big(V_M \mathrm{msgn}\big(S_M U_M^\top G(ZZ^\top)^{-1}\big)\Big) \\
&= 
U_M S_M \mathrm{msgn}\Big(S_M U_M^\top G(ZZ^\top)^{-1}\Big),
\end{align*}
which matches \eqref{eq:factorized-polar-qstar}.
\end{proof}

Motivated by this identity, we explored practical estimators of ${M}$ in benchmark training runs, including momentum-based and diagonal variants derived from update statistics. However, these non-isotropic variants did not yield consistent improvements over Newton--Muon.

\paragraph{Keeping a momentum buffer of the actual update.}
We maintain a momentum buffer of the actual parameter updates, denoted by ${M}$.
Heuristically, if the update directions tend to point toward an optimum $W^\star$, then this buffer can be viewed as a coarse proxy of
${W}-W^\star$.
In practice, we formed two candidate directions
\[
{Q}_{I}
=
\mathrm{msgn} \big({G}({Z}{Z}^\top)^{-1}\big),
\qquad
{Q}_{M}
=
{M} \mathrm{msgn} \Big({M}^\top {G}({Z}{Z}^\top)^{-1}\Big),
\]
and then combined them by a convex average after normalizing their scales
\[
{Q}
=
(1-\lambda) \frac{{Q}_{I}}{\|{Q}_{I}\|_F}
+
\lambda \frac{{Q}_{M}}{\|{Q}_{M}\|_F},
\qquad \lambda\in[0,1].
\]
Despite the motivation from Proposition~\ref{prop:factorized-polar-SigmaW}, this averaging did not yield consistent gains, and typically underperformed
the Newton--Muon update ${Q}_{I}$.

\paragraph{Using diagonal estimation of the covariance from update statistics.}
We also tried a computationally cheaper approximation that estimates a diagonal ${\Sigma}_{{W}}$ from the second moment of the actual update directions. Concretely, for each matrix parameter ${W}\in\mathbb{R}^{m\times n}$ we maintain a state vector
$\boldsymbol{u}\in\mathbb{R}^{m}$ that tracks the row-wise magnitudes of
the applied Muon directions via an EWMA:
\[
\boldsymbol{u}\leftarrow \beta \boldsymbol{u}+(1-\beta) \big\|{Q}\big\|_{\text{row}},
\]
where $\|\cdot\|_{\text{row}}$ denotes the $\ell_2$ norm along the column dimension.
We then form a damped diagonal estimate of ${\Sigma}_{{W}}$ using the squared row magnitudes,
\[
\boldsymbol{u}^{\odot 2}\ \approx\ \mathrm{diag}({\Sigma}_{{W}}),
\qquad
{M}^2
=
\mathrm{diag} \Big(\lambda \boldsymbol{u}^{\odot 2} + (1-\lambda)\overline{\boldsymbol{u}^{\odot 2}}\, \mathbf{1}_m\Big),
\]
where $\overline{\boldsymbol{u}^{\odot 2}}$ is the mean of the entries of $\boldsymbol{u}^{\odot 2}$, $\lambda\in[0,1]$ is a damping coefficient, and $\mathbf{1}_m\in\mathbb{R}^m$ is the all-ones vector. Let ${M}$ be the positive diagonal square root of ${M}^2$. The resulting update takes the form
\[
{Q}_{M}
\propto
{M}
\mathrm{msgn} \Big({M}{G}({Z}{Z}^\top)^{-1}\Big).
\]
Despite being inexpensive, this diagonal ${\Sigma}_{{W}}$ variant did not outperform Newton--Muon in our experiments and in some cases performed worse than standard Muon.

\paragraph{Conclusion.}
Overall, if the update direction is systematically misaligned with ${W}-W^\star$, then the resulting estimate of ${\Sigma}_{{W}}$ inherits this bias, and since the next update direction depends on this estimate, the bias can reinforce itself, leading to slower training.
By using an isotropic proxy (i.e., ${\Sigma}_{{W}}\propto {I}_m$), we avoid imposing assumptions about the unknown distribution of displacements while still capturing the most robust geometric component of the update through the matrix sign. This helps explain both the computational simplicity and the stability of Newton--Muon.

\end{document}